\documentclass[12pt,a4paper,twoside,reqno]{amsart}

\usepackage{tikz}
\usepackage{amsmath,amscd}
\usepackage{amssymb,amsfonts,mathrsfs}
\usepackage[driver=pdftex,margin=3cm,heightrounded=true,centering]{geometry}
\usepackage{math50}
\usepackage{JK}
\usepackage[colorlinks=true,linkcolor=blue,citecolor=blue]{hyperref}

\begin{document}
\title[On modular index theory]{On modular semifinite index theory}

\author{Jens Kaad}
\address{Institut de Math\'ematiques de Jussieu,
Universit\'e de Paris VII,
175 rue du Chevaleret,
75013 Paris,
France}
\email{jenskaad@hotmail.com}


%
%
%
\thanks{The author was supported by the Carlsberg Foundation}
\subjclass[2010]{58B34; 58B32, 46L80, 46L51, 19K56}
\keywords{Modular spectral triple, Chern characters, reduced twisted cyclic theory, twisted index pairing, derived $L^p$-spaces, quantum SU(2).}

\begin{abstract}
We propose a definition of a modular spectral triple which covers existing examples arising from KMS-states, Podle\'s sphere and quantum SU(2). The definition also incorporates the notion of twisted commutators appearing in recent work of Connes and Moscovici. We show how a finitely summable modular spectral triple admits a twisted index pairing with unitaries satisfying a modular condition. The twist means that the dimensions of kernels and cokernels are measured with respect to two different but intimately related traces. The twisted index pairing can be expressed by pairing Chern characters in reduced versions of twisted cyclic theories. We end the paper by giving a local formula for the reduced Chern character in the case of quantum SU(2). It appears as a twisted coboundary of the Haar-state. In particular we present an explicit computation of the twisted index pairing arising from the sequence of corepresentation unitaries. As an important tool we construct a family of derived integration spaces associated with a weight and a trace on a semifinite von Neumann algebra.
\end{abstract}

\maketitle
\tableofcontents

\section{Introduction}
Let us consider the core object in noncommutative geometry: A spectral triple $(\C A, \C H,D)$. We recall that $\C A$ is a (unital) $*$-algebra which is represented on a (possibly graded) Hilbert space $\C H$ whereas $D$ is an unbounded selfadjoint operator with compact resolvent and such that commutators with algebra elements are bounded, see \cite{Con:NCG,CoMo:LIF}. In recent work of Connes and Moscovici an extension of this notion was proposed, \cite[Definition 3.1]{CoMo:TST}. Instead of bounded straight commutators one allows for bounded twisted commutators where the twist is given by a (regular) automorphism $\te \in \au(\C A)$ of the $*$-algebra. Thus, the commutator relation now becomes $[D',x]_\te = D' x - \te(x) D' \in \C L(\C H)$. As a guideline for this extension one can consider the perturbed spectral $(\C A,\C H,D')$ obtained from the spectral triple $(\C A,\C H,D)$ and a positive invertible element $g \in \C A$ by replacing $D$ by $D':= \ov{g D g}$. The twisting automorphism is then defined by $\te(x) := g^2 x g^{-2}$, $x \in \C A$. The value at zero of zeta functions and the Ray-Singer analytic torsion have been computed for perturbed spectral triples of the above form for the noncommutative torus in \cite{CoMo:MCT}. See also \cite{FaKh:SCN}.

A simple modification of the above idea motivates to some extend the definition of a modular spectral triple. Indeed, one could ask the question:
\\

\emph{What happens if the inverse of the perturbation factor $g^2$ is an unbounded operator?}
\\

Let us be more precise. Thus, let $\De : \sD(\De) \to \C H$ be a bounded selfadjoint, positive and injective diagonal operator. Associated with $\De$ we then have the one-parameter group of automorphisms $\{\si_t\}$ defined by
\[
\si_t(T) := \De^{it} T \De^{-it} \q T \in \C L(\C H)
\]
as well as the fixed point von Neumann algebra 
\[
\C L(\C H)^\si := \{T \in \C L(\C H) \, | \, \si_t(T) = T \, , \, t \in \rr\}.
\]
As compatibility relations between the spectral triple $(\C A,\C H,D)$ and $\De$ we will then require that:
\begin{enumerate}
\item The one-parameter group of automorphisms group $\{\si_t\}$ restricts to an automorphism group of $\C A$ and each element in $\C A$ is analytic for $\{\si_t\}$.
\item The Dirac operator $D$ is affiliated with the fixed point von Neumann algebra $\C L(\C H)^\si$.
\end{enumerate}

We can then form the Dirac operator $D' := \ov{D \De}$ and define the perturbed spectral triple $(\C A,\C H,D')$. It can be proved that the twisted commutator $[D',x]_{\si_{-i}}$ extends to a bounded operator where $\si_{-i} \in \au(\C A)$ denotes the analytic extension of the one-parameter group of automorphisms at $-i \in \cc$. However, it is also not hard to see that the resolvent $(i + D')^{-1}$ need no longer be compact with respect to the operator trace due to the unboundedness of the operator $\De^{-1}$. In order to solve this problem one can instead use the weight $\phi := \T{Tr}(\De \cd) : \C L(\C H)_+ \to [0,\infty]$ to measure the growth of the resolvent. Indeed, this weight descends to a semifinite normal faithful trace $\wit \phi$ on the fixed point von Neumann algebra and we get that
\[
(i + D')^{-1} \in \C K\big(\C L(\C H)^\si,\wit \phi\big).
\]
Here $\C K(\C L(\C H)^\si,\wit \phi)$ denotes the $C^*$-algebra of compacts associated with the fixed point von Neumann algebra and the trace $\wit \phi$.

The perturbed spectral triple $(\C A,\C H,D')$ is an example of a \emph{modular spectral triple}.

Another motivating example comes from the work of Carey, Neshveyev, Nest and Rennie on the modular index theory associated with a KMS-state on a $C^*$-algebra, see \cite{CNNR:TEK}. Under some extra assumptions these authors obtain a triple $(\C A, \C N,D)$ where $\C N$ is a semifinite von Neumann algebra equipped with both a weight $\phi_D$ and a trace $\T{Tr}_\tau$. In this example the straight commutators are bounded but the resolvent again only lies in the compacts $\C K(\C N^\si,\wit \phi_D)$ of the centralizer of the weight $\phi_D$.
%

Recent analysis of the spectral triple over the Podle\'s sphere, as it appears in \cite{DaSi:DPQ} for example, reveals that the introduction of extra unbounded factors in local Hochschild cocycles and zeta functions gives rise to new interesting invariants of the coordinate algebra, \cite{KrEl:RFP,NeTu:LQS,ReSe:TLP}. These changes can be captured by interpreting the spectral triple in question as a modular spectral triple with respect to different weights.
%

Finally, together with Roger Senior, we discovered an interesting triple $(\C A(SU_q(2)),\C H,D_q)$ over quantum $SU(2)$, see \cite{KaaSen:TSQ}. This triple also exhibits a "modular" behaviour and does indeed fit into the general context of the present paper.

We will give a precise definition of a modular spectral triple in Section \ref{S:mss} of this paper. As the examples indicate the main features are as follows:
\begin{enumerate}
\item The commutator condition is replaced by a twisted commutator condition.
\item The growth of the resolvent is measured with respect to a weight instead of a trace.
\end{enumerate}

With the concept of a modular spectral triple in hand a natural question is:
\\

\emph{How to construct the Chern character of a modular spectral triple?}
\\

An approach to this question consists of looking at the phase $F = D|D|^{-1}$ of the Dirac operator $D$ and try to mimic the construction of the Chern-Connes character as exposed in \cite[IV.1.$\al$]{Con:NCG} while replacing the trace with the weight. Thus, the Chern character should be defined by
\[
(x_0,\ldots,x_n) \mapsto \frac{1}{2}\phi(F[F,x_0] \clc [F,x_n]) \q x_0,\ldots,x_n \in \C A.
\]
where $\phi : \C N_+ \to [0,\infty]$ is a weight on a semifinite von Neumann algebra. Under the extra condition of twisted Lipschitz regularity it can be seen that the commutator $[F,x]$ should have the same growth properties as the resolvent $(i + D)^{-1}$. However, we are faced with the problem of phrasing these growth properties in an appropriate way. Indeed, the commutator $[F,x]$ need not lie in the fixed point von Neumann algebra and Segal's theory of noncommutative $L^p$-spaces is therefore not available.

We present a solution to this problem by introducing a family of Fr\'echet spaces $\{L^p(\phi,\tau)\}_{p \in [1,\infty)}$ associated with a pair consisting of a weight $\phi$ and a trace $\tau$ on a semifinite von Neumann algebra $\C N$. Notably these derived $L^p$-spaces satisfy the H\"older type inclusions
\[
\sw{p} \cd \sw{q} \su \sw{r} \q 1/p + 1/q = 1/r
\]
and furthermore, we have dense subspaces $\{\sL^p(\phi,\tau)\}_{p \in [1,\infty)}$ such that $\sL^p(\phi,\tau) \su \sL^q(\phi,\tau)$ whenever $p \leq q$. Finally, the weight $\phi$ induces a twisted trace $\phi : L^1(\phi,\tau) \to \cc$ where the twist is given by the modular automorphism group of $\phi$ at the imaginary unit.

The basic idea is that an element $x$ in a von Neumann algebra $\C N$ should be \emph{derived $p$-summable} when the product $\De^{1/p} x$ is $p$-summable with respect to the trace $\tau : \C N_+ \to [0,\infty]$. Here $\De$ denotes the Radon-Nikodym derivative of the weight $\phi$ with respect to the trace $\tau$. 
%

As we have hinted at above the notion of derived summability allows us to construct the Chern character of a modular spectral triple. We state the result as a theorem. It can be obtained by combining Proposition \ref{p:cht} and Proposition \ref{p:modfre} of the present paper.

\begin{theorem}
Suppose that $\C D = (\C A,\C N,D)$ is a unital $p$-summable Lipschitz regular modular $\te$-spectral triple w.r.t. the weight $\phi : \C N_+ \to [0,\infty]$. Let $F = D|D|^{-1}$ denote the phase of $D$. Then the Chern character
\[
\T{Ch}_\phi^n(\C D) : (x_0,\ldots,x_n) \mapsto \frac{1}{2}\phi(\ga^{n+1}F[F,x_0]\clc [F,x_n]) \q x_0,\ldots,x_n \in \C A
\]
is a well-defined twisted cyclic cocycle with twist given by the modular automorphism $\si_i \in \au(\C A)$ at the imaginary unit. Here $\ga$ is the grading operator in the even case and $n$ is the smallest integer of the same parity as the spectral triple with $n+1 \geq p$.
\end{theorem}

It is desirable to incorporate certain invariance properties of the modular spectral triple into the above theorem. In particular it could turn out that the commutator $[D,x] = 0$ was trivial for all elements in a subalgebra $\C B \su \C A^\si$ of the fixed point algebra of the modular group of automorphisms $\{\si_t\}$. We cover this situation by introducing a reduced version of the twisted cyclic theory. The main change consists of replacing the algebra $\C A$ with the vector space quotient $\C A/\C B$ in the defining complexes. Surprisingly enough this gives rise to well-defined chain and cochain complexes. 
%

In order to extract interesting numerical information from the algebra $\C A$ using the Chern character of a modular spectral triple we need to construct a Chern character from a class of objects in $\C A$ with values in (reduced) twisted cyclic homology. We will restrict ourselves to the odd case. A relevant class of objects would then a priori be the unitaries over the algebra. However the presence of the twisting automorphism $\si_i \in \au(\C A)$ forces us to impose an extra condition on the unitaries under consideration. In this paper we look at \emph{right $\C B$-modular} unitaries $u \in \C U_k(\C A)$. This means that the product
\[
u^* \cd \si_z(u) \in GL_k(\C B)
\]
is invertible over the subalgebra $\C B \su \C A^\si$ for all $z \in \cc$. It is then immediate that the Chern character
\[
\T{Ch}_{2n-1}^\phi(u) := \T{TR}([u] \ot_{M_k(\C B)} [u^*]
\ot_{M_k(\C B)} \ldots \ot_{M_k(\C B)} [u] \ot_{M_k(\C B)} [u^*]) \in Z^\la_{2n-1}(\C A/\C B, \si_i)
\]
defines a $\C B$-reduced twisted cyclic homology class for all $n \in \nn$ where $\T{TR}$ denotes a reduced twisted version of the generalized trace.

The final main question which we address in this paper can now be phrased:
\\

\emph{Is it possible to give an index theoretical interpretation of the pairing of Chern characters in reduced twisted cyclic theory?}
\\

Let us form the abstract Toeplitz operator $PuP \in \C N$ where $P = E_D([0,\infty))$ denotes the spectral projection of $D$ associated with the halfline $[0,\infty)$. The pairing of Chern characters should then measure the difference in size of the kernel projection $K_{PuP}$ and the cokernel projection $K_{Pu^*P}$. However, the asymmetric behaviour of the weight and the modular unitary implies that these two projections lie in the trace ideal of two \emph{different} semifinite von Neumann algebras. To be more precise, let us form the weight $\psi : = \phi(g_{-i} \, \cd  \,)$ where $g_{-i} := u^* \si_{-i}(u)$ is a positive element in $GL_k(\C B)$. The kernel projection $K_{PuP}$ then lies in the trace ideal $\sL^1(\C N^\psi,\wit \psi)$ associated with the centralizer of the weight $\psi$ whereas the cokernel projection $K_{Pu^*P}$ lies in the trace ideal $\sL^1(\C N^\phi,\wit \phi)$ associated with the centralizer of the weight $\phi$. The relation between the pairing of Chern characters in $\C B$-reduced twisted cyclic theory and the twisted (or modular) index of the abstract Toeplitz operator $PuP$ can now be stated:

\begin{theorem} 
Let $u \in \C U_k(\C A)$ be a right $\C B$-modular unitary and let $\C D = (\C A,\C N,D)$ be a $\C B$-invariant odd unital $2n$-summable Lipschitz regular modular $\te$-spectral triple with respect to $(\phi,\tau)$. We then have the identity
\[
\T{Ind}_\phi(PuP) := \psi(K_{PuP}) - \phi(K_{Pu^* P}) = \frac{(-1)^{n+1}}{2^{2n-1}}\binn{\T{Ch}^\phi_{2n-1}(u),\T{Ch}_\phi^{2n-1}(\C D)}
\]
between the twisted index pairing and the pairing of reduced Chern characters.
\end{theorem}

The present paper concludes with a complete computation of the modular index pairing in the case of the modular spectral triple $\C D_q = (\C A(SU_q(2)),\C H,D_q)$ over quantum $SU(2)$ defined in \cite{KaaSen:TSQ}. The main result is that the reduced Chern character $\ch{1}{\C D_q} \in H^1_\la(\C A/\cc,\si_i)$ is represented by the reduced twisted cocycle
\[
(x,y) \mapsto C \cd h(b_{\si_i}(x,y))
\]
where $C> 0$ is an explicit constant and $h : SU_q(2) \to \cc$ denotes the Haar-state. In particular we get that the Chern character is trivial in non-reduced twisted cyclic cohomology. As a consequence we have the very simple formula for the modular index pairings associated with the sequence of corepresentation unitaries:
\[
\T{Ind}_\phi(Pu^lP) = C/2 \cd \big( (2l+1) - [2l+1]_{q^{1/2}}\big) \q l \in \frac{1}{2} \nn \cup \{0\}.
\]
Here $[2l+1]_{q^{1/2}} := \frac{q^{l+1/2} - q^{-(l + 1/2)}}{q^{1/2} - q^{-1/2}}$ is a $q^{1/2}$-integer. 
%

\subsection{Acknowledgements}
I would like to thank Matthias Lesch, Ryszard Nest, Adam Rennie and Roger Senior for many stimulating discussions. I should also mention that I have borrowed a fair deal of my latex setup from Matthias Lesch. I am of course grateful to him. I would also like to thank Etienne Blanchard and George Skandalis for helpful comments. Finally, I would like to thank the Danish Carlsberg Foundation for having supported the research for the present paper.

\section{Derived $L^p$-spaces}\label{S:der}
Let $\C N$ be a semifinite von Neumann algebra with a fixed semifinite normal faithful trace $\tau : \C N_+ \to [0,\infty]$. Furthermore, let $\phi : \C N_+ \to [0,\infty]$ be a fixed semifinite normal faithful weight. We will use the notation $\{\si_t\}$ for the modular group of automorphisms associated with $\phi$.

In this section we will introduce a family $\{L^p(\phi,\tau)\}_{p \in [1,\infty)}$ of Fr\'echet spaces. Each of the Fr\'echet spaces $L^p(\phi,\tau)$ is a bimodule over the analytic operators with respect to the modular automorphism group $\{\si_t\}$. Furthermore, we have the H\"older type inclusions
\[
L^p(\phi,\tau) \cd L^q(\phi,\tau) \su L^r(\phi,\tau) \q 1/r = 1/p + 1/q.
\]
Finally, we get that the weight $\phi$ defines a twisted trace on $L^1(\phi,\tau)$ where the twist is given by the analytic extension of the modular automorphism group at $i$. The bounded analytic operators in the derived $L^p$-space will be denoted by $\sL^p(\phi,\tau)$ and we have the inclusion $\sL^p(\phi,\tau) \su \sL^q(\phi,\tau)$ whenever $p \leq q$. Even though the definition of the Fr\'echet spaces $L^p(\phi,\tau)$ is still in a preliminary state, they will play an important role for our development of a twisted index pairing. For convenience we will refer to the Fr\'echet space $L^p(\phi,\tau)$ as the derived $L^p$-space of order $p$.
%
%

Let us recall the definition of the noncommutative $L^p$-space associated with the trace $\tau$, see \cite{Seg:NEA}. Let $p \in [1,\infty)$ and let $\sL^p(\tau)$ denote the set of elements $x \in \C N$ such that $\tau(|x|^p) < \infty$. The set $\sL^p(\tau)$ becomes a normed space when equipped with the vector space structure from $\C N$ and the norm $\| \cd \|_p : \sL^p(\tau) \to [0,\infty)$, $\|x\|_p := \tau(|x|^p)^{1/p}$.

\begin{dfn}
By the \emph{noncommutative $L^p$-space} associated with the trace $\tau : \C N_+ \to [0,\infty]$ we will understand the Banach space obtained as the completion of $\sL^p(\tau)$ with respect to the norm $\|\cd\|_p$. The noncommutative $L^p$-space will be denoted by $L^p(\tau)$.
\end{dfn}

We remark that the noncommutative $L^p$-spaces have an equivalent definition in terms of singular numbers of $\tau$-measurable operators, see \cite{FacKos:GMO}.

The noncommutative $L^2$-space $L^2(\tau)$ is in fact a Hilbert space with inner product induced by $\inn{x,y} = \tau(x^* y)$ for all $x,y \in \sL^2(\tau)$. There is a faithful normal representation $\pi : \C N \to \C L(L^2(\tau))$ of the von Neumann algebra $\C N$ on $L^2(\tau)$ as left multiplication operators. We will identify $\C N$ with its image $\pi(\C N)$ and use the notation $\C H := L^2(\tau)$. Note that $\pi(\C N)' = J \pi(\C N) J$ where $J : L^2(\tau) \to L^2(\tau)$ is induced by the adjoint operation.

Now, by \cite[Theorem 5.12]{PeTa:RNN} there exists a unique selfadjoint positive non-singular operator $\De : \sD(\De) \to \C H$ such that 
\[
\phi(x) = \lim_{n \to \infty} \tau(\De_{1/n}^{1/2} x \De_{1/n}^{1/2}) \q \De_{1/n} := \De(1 + \De/n)^{-1}.
\]
We will refer to $\De$ as the \emph{Radon-Nikodym derivative}. The modular automorphism group of $\phi$ is then given by $\si_t(x) = \De^{it} x \De^{-it}$, $x \in \C N$, see \cite[Theorem 4.6]{PeTa:RNN}.


\begin{dfn}
An element $x \in \C N$ is said to be \emph{analytic} when the map $t \mapsto \si_t(x)$ extends to an entire map $\cc \to \C N$, $z \mapsto \si_z(x)$.
\end{dfn}

We recall that the analytic elements form a $\si$-weakly dense $*$-subalgebra of $\C N$ and that $\{\si_z\}_{z \in \cc}$ is a complex parameter group of algebra automorphisms of $\C N_{\T{an}}$ with $\si_z(x^*)^* = \si_{\ov z}(x)$. The analytic operators become a Fr\'echet $*$-algebra when equipped with the fundamental system of semi-norms $\{\|\cd\|_n\}$, $\|x\|_n := \T{sup}_{t \in [-n,n]}\|\si_{it}(x)\|$.
%
%
%

Let $p \in [1,\infty)$. We let $\sL^p(\phi,\tau)$ denote the set of analytic operators $x \in \C N_{\T{an}}$ such that
\begin{enumerate}
\item The closed unbounded operator
\[
\De^{1/p} \si_{it}(x) : \sD(\De^{1/p} \si_{ir}(x)) \to \C H
\]
lies in the noncommutative $L^p$-space $L^p(\tau)$ for all $r \in \rr$.
\item The map $t \mapsto \|\De^{1/p} \si_{it}(x)\|_p$ is locally bounded.
\end{enumerate}

We note that $\sL^p(\phi,\tau)$ is a vector space with a countable family of norms $\{\|\cd\|_{p,n}\}_{n \in \nn}$ defined by
\[
\|x\|_{p,n} = \sup_{t \in [-n,n]}\|\De^{1/p}\si_{it}(x)\|_p \in [0,\infty).
\]
We let $d_p$ denote the metric on $\sL^p(\phi,\tau)$ defined by
\[
d_p(x,y) = \sum_n 2^{-n} \frac{\|x-y\|_{p,n}}{1 + \|x-y\|_{p,n}}.
\]

\begin{dfn}
By the \emph{derived $L^p$-space} associated with the pair $(\phi,\tau)$ we will understand the completion of $\sL^p(\phi,\tau)$ with respect to the metric $d_p$. The derived $L^p$-space will be denoted by $L^p(\phi,\tau)$.
\end{dfn}

The derived $L^p$-space $L^p(\phi,\tau)$ is a Fr\'echet space with a countable family of norms given by $\{\|\cd\|_{p,n}\}_{n \in \nn}$.



\begin{lemma}
The derived $L^p$-space comes equipped with an anti-linear map $* : L^p(\phi,\tau) \to L^p(\phi,\tau)$ with square equal to the identity and with
$\|x^*\|_{p,n} \leq \|x\|_{p,n + \lceil 1/p \rceil}$ for all $x \in L^p(\phi,\tau)$. The operator $*$ is induced by the adjoint operation in $\C N$.
\end{lemma}
\begin{proof}
We only need to verify the claim on the dense subspace $\sL^p(\phi,\tau)$. Thus, let $x \in \sL^p(\phi,\tau)$. We then have that $(\De^{1/p} x)^* = \ov{x^* \De^{1/p}} \su \De^{1/p} \si_{i/p}(x^*)$. In particular we get by $\tau$-measurability that $(\De^{1/p} x)^* = \De^{1/p} \si_{i/p}(x^*)$, see \cite[Proposition 12]{Ter:LNA}. But this implies that $\De^{1/p}\si_{it}(x^*) \in L^p(\tau)$ for all $t \in \rr$. Furthermore, we get the inequality
$\|x^*\|_{p,n} \leq \|x\|_{p,n + \lceil 1/p \rceil}$ since $* : L^p(\tau) \to L^p(\tau)$ is isometric.
\end{proof}

It follows from the proof given above that the dense subspace $\sL^p(\phi,\tau)$ is invariant for the operator $* : L^p(\phi,\tau) \to L^p(\phi,\tau)$.

\begin{lemma}\label{l:ide}
The derived $L^p$-space is a bimodule over $\C N_{\T{an}}$ and we have the inequality
\[
\|y \cd x \cd z\|_{p,n} \leq \|y\|_{n + \lceil 1/p \rceil} \cd \|x\|_{p,n} \cd \|z\|_n
\]
whenever $x \in L^p(\phi,\tau)$, $y,z \in \C N_{\T{an}}$ and $n \in \nn$. The module actions are induced by the product in $\C N$.
\end{lemma}
\begin{proof}
As before we may restrict our attention to the dense subspace $\sL^p(\phi,\tau)$. Clearly $\sL^p(\phi,\tau)$ is a right ideal in $\C N_{\T{an}}$. It is also a left ideal since $\sL^p(\phi,\tau)^* = \sL^p(\phi,\tau)$ and $\C N_{\T{an}}^* = \C N_{\T{an}}$. To prove the desired product inequality we start by noting that $\si_{-i/p}(y) \De^{1/p} x \su \De^{1/p}yx$. By $\tau$-measurability of $\De^{1/p}x$ we thus get the identity $\si_{-i/p}(y) \De^{1/p}x = \De^{1/p}yx$. We can then conclude that
\[
\|\De^{1/p} y x z \|_p \leq \|\De^{1/p}y x\|_p \|z\| \leq \|\si_{-i/p}(y)\| \cd \|\De^{1/p}x\|_p \cd \|z\|
\]
which proves the claim.
\end{proof}

It follows from the proof given above that the dense subspace $\sL^p(\phi,\tau)$ is an ideal in $\C N_{\T{an}}$.

We are now ready to prove the H\"older inequalities for the derived $L^p$-spaces. This will be very important for our construction of the twisted index pairing.

\begin{lemma}\label{l:hol}
Let $p,q \in [1,\infty)$. Let $x \in L^p(\phi,\tau)$ and let $y \in L^q(\phi,\tau)$ and suppose that $1/p + 1/q \leq 1$. Then we have a well-defined product $x \cd y \in L^r(\phi,\tau)$ where $1/p + 1/q = 1/r$. Furthermore, we have the H\"older inequality
\[
\|x y\|_{r,n} \leq \|x\|_{p,n + \lceil 1/q \rceil} \cd \|y\|_{q,n}
\]
for all $n \in \nn$.
\end{lemma}
\begin{proof}
We restrict our attention to the dense subspaces $\sL^p(\phi,\tau)$ and $\sL^q(\phi,\tau)$.

By \cite[Theorem 4.2]{FacKos:GMO} we have that the closure of the unbounded operator $\De^{1/p}\si_{-i/q}(x) \cd \De^{1/q}y$ lies in the $L^r$-space $L^r(\tau)$. However, using that $\si_{-i/q}(x) \De^{1/q}(\xi) = \De^{1/q} x (\xi)$ for all $\xi \in \sD(\De^{1/q})$ we get the inclusion $\De^{1/p} \si_{-i/q}(x) \cd \De^{1/q} y \su \De^{1/r} x y$. It then follows by $\tau$-measurability that we have the identity
\[
\De^{1/r} xy = \ov{\De^{1/p} \si_{-i/q}(x) \cd \De^{1/q} y} \in L^r(\tau).
\]
Now, using the H\"older inequality for $L^p$-spaces one more time we conclude that
\[
\| \De^{1/r} x y\|_r \leq \|\De^{1/p} \si_{-i/q}(x)\|_p \cd \|\De^{1/q} y \|_q
\]
which implies the stated H\"older inequality for the derived $L^p$-spaces.
\end{proof}

We remark that the bilinear map $L^p(\phi,\tau) \ti L^q(\phi,\tau) \to L^r(\phi,\tau)$ is induced by the product in $\C N$ and that the restriction to $\sL^p(\phi,\tau) \ti \sL^q(\phi,\tau)$ takes values in $\sL^r(\phi,\tau)$.

In the next lemma we study the question whether the derived $L^p$-space is included in the derived $L^q$-space when $p\leq q$. We are however only able to prove this property for the dense subspaces of analytic bounded operators and leave it as an open question whether the result is valid for the entire derived integration space.

\begin{lemma}\label{l:inc}
Let $p,q \in [1,\infty)$ with $p \leq q$. We then have the inclusion $\sL^p(\phi,\tau) \su \sL^q(\phi,\tau)$. Furthermore, we have the size estimate
\[
\|x\|_{q,n} \leq \|x\|_n^{1 - p/q} \cd \|x\|_{p,n}^{p/q}
\]
for all $n \in \nn$. 
%
\end{lemma}
\begin{proof}
Let $P := E([0,\la])$ denote the spectral projection associated with the Radon-Nikodym derivative $\De$ and the interval $[0,\la]$ for some $\la > 0$. We then have that $x \in \sL^q(\phi,\tau)$ if and only if $Px \in \sL^q(\phi,\tau)$ and $(1- P)x \in \sL^q(\phi,\tau)$. 

To continue we note that
\[
\De^{1/q}\si_{it}((1 -P)x) 
= \De^{1/q - 1/p}(1 - P)\De^{1/p}\si_{it}((1- P)x).
\]
This shows that $(1 - P) x \in \sL^q(\phi,\tau)$ since $\De^{1/q - 1/p}(1 - P) \in \C N$. Furthermore, we have the norm estimate
\[
\|(1 - P)x\|_{q,n} 
\leq \|\De^{1/q - 1/p} (1 - P)\|\cd \|x\|_{p,n}
\leq \la^{1/q - 1/p} \cd \|x\|_{p,n}
\]
for all $n \in \nn$.

Next, we consider the bounded analytic operator $\De^{1/q}Px$. Since $\De^{2/q} P$ is a bounded positive operator and $t \mapsto t^{q/p}$ is a continuous increasing convex function on $[0,\infty)$ which vanish at zero we get the inequality
\[
\begin{split}
\mu_t\big( ( x^* \De^{2/q}P x)^{q/p} \big) & = 
\|x\|^{2 q/p}\mu_t\big( (x^* \De^{2/q}P x /\|x\|^2 )^{q/p} \big)
\leq \|x\|^{2 q/p} \mu_t( x^* \De^{2/p}P x / \|x\|^2 ) \\
& = \|x\|^{2 (q / p - 1)} \mu_t(x^* \De^{2/p}P x)
\end{split}
\]
of singular values. See \cite[Lemma 4.5]{FacKos:GMO} and \cite[Theorem 11]{BroKos:JSN}. In particular we get that
\[
\begin{split}
\tau( (x^* \De^{2/q}P x)^{q/2})
& = \int \mu_t\big( (x^* \De^{2/q}P x)^{q/p} \big)^{p/2} dt
\leq \|x\|^{q - p} \int \mu_t(x^* \De^{2/p}Px)^{p/2} dt \\
& = \|x\|^{q - p} \|\De^{1/p}P x\|_p^p
< \infty.
\end{split}
\]
This shows that $Px \in \sL^q(\phi,\tau)$ and that we have the norm estimate
\[
\begin{split}
\|Px\|_{q,n} & = \sup_{t \in [-n,n]} \|\De^{1/q}\si_{it}(Px)\|_q
= \sup_{t \in [-n,n]}\tau( (\si_{it}(x)^* \De^{2/q}P \si_{it}(x))^{q/2})^{1/q} \\
& \leq \sup_{t \in [-n,n]} \|\si_{it}(x)\|^{1 - p/q} \cd
\|\De^{1/p} \si_{it}(Px)\|_p^{p/q}
\leq \|x\|_n^{1 - p/q} \cd \|x\|_{p,n}^{p/q}
\end{split}
\]
for all $n \in \nn$.

We have thus proved that $x \in \sL^q(\phi,\tau)$ and furthermore, by the Minkowsky inequality we get that
\[
\|x\|_{q,n} \leq \|Px\|_{q,n} + \|(1-P)x\|_{q,n}
\leq \|x\|_n^{1 - p/q} \cd \|x\|_{p,n}^{p/q} + \la^{1/q - 1/p} \cd \|x\|_{p,n}
\]
for all $\la > 0$. But this proves the desired size estimate as well.
%
\end{proof}

In the following we will be concerned with the construction of a twisted trace $\phi : L^1(\phi,\tau) \to \cc$ which is induced by the weight $\phi : \C N_+ \to \cc$. 

We note that the complex parameter group of automorphisms $\si_z : \sL^p(\phi,\tau) \to \sL^p(\phi,\tau)$ extends to a complex parameter group of automorphisms on the derived $L^p$-spaces. Furthermore, we have the size estimate $\|\si_z(x)\|_{p,n} \leq \|x\|_{p,n + \lceil \T{Im}(z) \rceil}$ where $\T{Im}(z)$ refers to the imaginary part of $z \in \cc$. 

We define the functional $\wit \phi : \sL^1(\phi,\tau) \to \cc$ by the formula
\[
\wit \phi(x) = \tau(\De x) \q x \in \sL^1(\phi,\tau).
\]
We then have that $|\wit \phi(x)| \leq \|x\|_{1,n}$ for all $n \in \nn$. In particular, the functional $\wit \phi$ extends by continuity to the derived $L^1$-space $L^1(\phi,\tau)$.

\begin{lemma}\label{l:conv}
Let $x \in \sL^1(\phi,\tau)$. We then have the convergence result
\[
\wit \phi(x) = \lim_{n \to \infty} \tau(\De_{1/n} x)
\, , \q \De_{1/n} := \De (1 + \De/n )^{-1}.
\]
\end{lemma}
\begin{proof}
Using the tracial property of the functional $\tau : L^1(\tau) \to \cc$ we get that $\tau(\De_{1/n} x) = \tau(\De x (1 + \De/n)^{-1})$. It then follows by the normality of the trace that $\lim_{n \to \infty} \tau(\De x (1 + \De/n)^{-1}) = \tau(\De x) = \wit \phi(x)$.
\end{proof}
%
%

\begin{lemma}\label{l:invar}
The functional $\wit \phi : L^1(\phi,\tau) \to \cc$ is invariant under the complex parameter group of automorphisms $\{\si_z\}$. Thus, $\wit \phi(\si_z(x)) = \wit \phi(x)$ for all $x \in L^1(\phi,\tau)$ and all $z \in \cc$.
\end{lemma}
\begin{proof}
As usually we may restrict our attention to the analytic operators $\sL^1(\phi,\tau)$. Let $n \in \nn$. Define $\wit \phi_n(y) := \tau(\De_{1/n}y)$ for all $y \in \sL^1(\phi,\tau)$. By Lemma \ref{l:conv} we only need to show that $\wit \phi_n(x) := \wit \phi_n(\si_z(x))$.

We start by proving that the map $z \mapsto \wit \phi_n(\si_z(x))$ is holomorphic.

Since $\tau$ is semifinite and normal we can find positive normal functionals $\{\om_i\}_{i \in I}$ such that $\tau(y) = \sum_{i \in I} \om_i(y)$, $y \in \sL^1(\tau)$ where the sum is absolutely convergent with
\[
\sum_{i \in I} |\om_i(y)| 
\leq \tau(|\T{Re}(y)|) + \tau(|\T{Im}(y)|)
\leq 2 \|y\|_1.
\]
In particular we get the identity $\wit \phi_n(\si_z(x)) = \sum_{i \in I} \om_i(\De_{1/n}\si_z(x))$. Now, each of the terms in the sum is holomorphic by the normality of the functionals. Furthermore, we have the following local uniform convergence
\[
\T{sup}_{z \in D_m} \sum_{i \in I} |\om_i(\De_{1/n} \si_z(x))|
\leq 2 \cd \T{sup}_{z \in D_m} \|\De_{1/n} \si_z(x)\|_1
\leq 2 \|x\|_{1,m}
\]
where $D_m \su \cc$ is a closed disc of radius $m \in \nn$ and center at the origin. But this implies that $z \mapsto \wit \phi_n(\si_z(x))$ is holomorphic.

The result of the lemma now follows by the uniqueness of holomorphic extensions by noting that $\wit \phi_n(\si_t(x)) = \wit \phi_n(x)$, $t \in \rr$.
\end{proof}

We can now prove that the functional $\wit \phi : L^1(\phi,\tau) \to \cc$ agrees with the weight $\phi : \C N_+ \to [0,\infty]$ whenever they are both defined.

\begin{lemma}\label{l:weitwt}
Let $x \in \sL^1(\phi,\tau)$ be positive. We then have the identity
\[
\phi(x) = \wit \phi(x).
\]
\end{lemma}
\begin{proof}
We start by recalling that $\phi(x) = \lim_{n \to \infty} \tau(\De_{1/n}^{1/2} x \De_{1/n}^{1/2})$.

Let $n \in \nn$. We then have the identity
$\De_{1/n}^{1/2}x \De_{1/n}^{1/2} = \De_{1/n}^{1/2} \De^{1/2} \si_{i/2}(x) (1 + \De/n)^{-1/2}$. And it thus follows by the tracial property of $\tau$ that
\[
\tau(\De_{1/n}^{1/2} x \De_{1/n}^{1/2}) = \tau(\De_{1/n} \si_{i/2}(x)).
\]
The result of the lemma is now a consequence of Lemma \ref{l:conv} and Lemma \ref{l:invar}.
\end{proof}
%

From now on we will omit the $\wit \cd$ from the functional $\wit \phi = \tau(\De \cd) : L^1(\phi,\tau) \to \cc$. We end this section by showing that $\phi$ is a twisted trace on $L^1(\phi,\tau)$.

\begin{prop}\label{p:twt}
Let $x \in L^p(\phi,\tau)$ and $y \in L^q(\phi,\tau)$ with $1/p + 1/q = 1$. Then the operators $x \cd y$ and $\si_i(y) \cd x$ lie in the derived $L^1$-space and we have the identity
\begin{equation}\label{eq:twt}
\phi(x y) = \phi(\si_i(x)y).
\end{equation}
\end{prop}
\begin{proof}
As usually, we restrict our attention to the dense subspaces of bounded analytic operators. We start by noting that $\De x y = \ov{\De^{1/p}\si_{-i/q}(x) \cd \De^{1/q} y}$. Using the tracial property of $\tau : L^1(\tau) \to \cc$ we can then compute as follows
\[
\begin{split}
\phi(xy) & = \tau(\ov{\De^{1/p}\si_{-i/q}(x) \cd \De^{1/q} y}) 
= \tau(\ov{\De^{1/q} y \cd \De^{1/p}\si_{-i/q}(x) }) \\
& = \tau(\De \si_{i/p}(y) \si_{-i/q}(x))
= \phi(\si_{i/p}(y) \si_{-i/q}(x)).
\end{split}
\]
The result of the proposition is then a consequence of the invariance of $\phi$ under the automorphism $\si_{i/q} \in \au(L^1(\phi,\tau))$.
\end{proof}

\begin{remark}
The derived $L^p$-spaces considered in this section are related to the $L^p$-spaces $L_{p,\al}(\phi)$, $\al \in [0,1]$ considered by Trunov and Sherstnev in \cite{TrSh:ITN}. An important difference is however our incorporation of the modular group of automorphisms in the definition $L^p(\phi,\tau)$. We believe that this is necessary in order to obtain many of the results proved in this section.

We also remark that our derived $L^p$-spaces are different from the spatial $L^p$-spaces or Haagerup $L^p$-spaces associated with the weight $\phi$ as defined in \cite{Con:SNA} and \cite{Haa:LAN,Ter:LNA}.
\end{remark}
%

\section{Reduced twisted cyclic theory}\label{S:cyc}
In this section we will develop a reduced version of the twisted cyclic theory as it appears in \cite{KuMuTu:DQT}.

Throughout this section we let $\C A$ be a unital algebra over $\cc$ which comes equipped with an automorphism $\si \in \au(\C A)$. Furthermore, we let $\C B \su \C A$ be a subalgebra such that $\si$ restricts to an automorphism $\si \in \au(\C B)$. The main idea is then to replace the algebra $\C A$ by the vector space quotient $\C A/\C B$ in the definition of the twisted cyclic complex. The corresponding theory allows us to incorporate certain invariance properties of a (modular) Fredholm module into the definition of the Chern character. See also the related concept of invariant cyclic cohomology as defined in \cite[Section 9]{Con:CQL}.

\subsection{Reduced homology}
Since the vector space quotient $\C A/\C B$ is a bimodule over $\C B$ we can form the tensor product
\[
(\C A/\C B)^{\ot_{\C B}(n+1)}:= \underbrace{(\C A/\C B) \ot_{\C B} \ldots \ot_{\C B} (\C A/\C B)}_{n+1}
\]
for each $n \in \nn \cup \{0\}$ which is again a bimodule over $\C B$.

For each $\om \in (\C A/\C B)^{\ot_{\C B}(n+1)}$ and each $b \in \C B$ we have the twisted commutator
\[
[\om,b]_\si = \om \cd b - \si(b) \cd \om.
\]
We let $\big[(\C A/\C B)^{\ot_{\C B}(n+1)},\C B \big]_\si$ denote the sub vector space of $(\C A/\C B)^{\ot_{\C B}(n+1)}$ generated by all twisted commutators of the above form. The notation $C_n(\C A/\C B)$ will refer to the vector space quotient
\[
C_n(\C A/\C B) := (\C A/\C B)^{\ot_{\C B}(n+1)} / \big[(\C A/\C B)^{\ot_{\C B}(n+1)},\C B \big]_\si.
\]

The twisted cyclic operator $\la : \C A^{\ot(n+1)} \to \C A^{\ot(n+1)}$ defined by $\la(a_0\olo a_n) := (-1)^n \si(a_n) \ot a_0 \olo a_n$ then descends to a cyclic operator $\la : C_n(\C A/\C B) \to C_n(\C A/\C B)$. Remark that
\[
\begin{split}
& \la(a_0 \olo a_{n-1} \ot ba_n) - \la(a_0 \olo a_{n-1}b \ot a_n) \\
& \q = (-1)^n \si(ba_n) \ot a_0 \olo a_{n-1} - (-1)^n \si(a_n) \ot a_0 \olo (a_{n-1} b) \\
& \q \in \big[(\C A/\C B)^{\ot_{\C B}(n+1)},\C B \big]_\si,
\end{split}
\]
thus the cyclic operator is only well-defined because we compute modulo twisted commutators with elements from $\C B$.

\begin{dfn}
By the $\C B$-reduced $\si$-twisted cyclic $n$-chains we will understand the vector space quotient
\[
C_n^\la(\C A/\C B,\si) := C_n(\C A/\C B) / \T{Im}(1 - \la).
\]
\end{dfn}

We define the twisted Hochschild boundary $b_\si : \C A^{\ot(n+1)} \to \C A^{\ot n}$ by the formula
\[
\begin{split}
b_\si : a_0 \olo a_n & \mapsto \sum_{i=0}^{n-1} (-1)^i a_0 \olo a_i \cd a_{i+1} \olo a_n \\
& \qq + (-1)^n \si(a_n) a_0 \ot a_1 \olo a_{n-1} \\
& \q = b'(a_0 \olo a_n)
+ (-1)^n \si(a_n) a_0 \ot a_1 \olo a_{n-1}
\end{split}
\]
where $b'$ is the usual boundary operator of the bar complex. It can then be proved that we have the identities
\begin{equation}\label{eq:tid}
b_\si^2 = 0 \q \T{and} \q (1- \la) b' = b_\si (1- \la)
\end{equation}
which entail that the twisted Hochschild boundary and the twisted cyclic chains form a chain complex. See \cite{KuMuTu:DQT}. The main point is now that the twisted Hochschild boundary induces a boundary map on the reduced twisted cyclic chains. Thus the twisted cyclic homology exists in a reduced version where the reduction is carried out with respect to any subalgebra which is preserved by the automorphism.

\begin{prop}\label{p:thb}
The twisted Hochschild boundary descends to a map
\[
b_\si : C_n^\la(\C A/\C B,\si) \to C_{n-1}^\la(\C A/\C B,\si).
\]
In particular we get a chain complex consisting of the reduced twisted cyclic chains and the twisted Hochschild boundary.
\end{prop}
\begin{proof}
We start by proving that the map
\begin{equation}\label{eq:1hb}
b_\si : \underbrace{\C A/\C B \ot_\cc \ldots \ot_\cc \C A/\C B}_{n+1} \to C_{n-1}^\la(\C A/\C B,\si)
\end{equation}
is well defined. Thus, let $a_0,\ldots,a_n \in \C A$ and assume that $a_i = b \in \C B$ for some $i \in \{0,\ldots,n\}$. We need to prove that $b_\si(a_0 \olo a_n) \sim 0$. Now, using the second identity of \eqref{eq:tid} we may assume that $i = 0$. We then have that
\begin{equation}\label{eq:hpb}
\begin{split}
& b_\si(b \ot a_1 \olo a_n) \\
& \q \sim  ba_1 \ot a_2 \olo a_n
+ (-1)^n \si(a_n)b \ot a_1 \olo a_{n-1} \\
& \q \sim (1- \la)(ba_1 \ot a_2 \olo  a_n) \sim 0.
\end{split}
\end{equation}
This proves the existence of the map in \eqref{eq:1hb}.

We continue by proving that the map
\begin{equation}\label{eq:2hb}
b_\si : C_n(\C A/\C B) \to C_{n-1}^\la(\C A/\C B,\si)
\end{equation}
is well-defined. Thus, let $a_0,\ldots,a_n \in \C A$ and let $b \in \C B$. We start by proving the equivalence
\[
b_\si(a_0 \olo a_i b \ot a_{i+1} \olo a_n) \sim
b_\si(a_0 \olo a_i \ot b a_{i+1} \olo a_n)
\]
for some $i \in \{0,\ldots,n-1\}$. Again, using the second identity of \eqref{eq:tid} we may assume that $i = 0$. But in this case the result is straightforward. We conclude by proving the equivalence
\[
b_\si(a_0 \olo a_{n-1} \ot a_n b) \sim
b_\si(\si(b)a_0 \ot a_1 \olo a_n).
\]
This follows from the computation
\[
\begin{split}
& b_\si(a_0 \olo a_{n-1} \ot a_n b) -
b_\si(\si(b)a_0 \ot a_1 \olo a_n) \\
& \q = \sum_{i=0}^{n-1} (-1)^i \big[a_0 \olo a_i a_{i+1} \olo a_n, b\big]_\si \\
& \q \sim 0.
\end{split}
\]
The result of the proposition now follows from the identities in \eqref{eq:tid}.
\end{proof}

We would like to warn the reader that the twisted Hochschild boundary $b_\si : \C A^{\ot (n+1)} \to \C A^{\ot n}$ does \emph{not} descend to a boundary map on the vector spaces $C_*(\C A/\C B)$. Indeed, the equivalence in \eqref{eq:hpb} is only valid because of the invariance under twisted cyclic permutations. In particular, the $\C B$-reduced twisted Hochschild homology is not well-defined as such.

\begin{dfn}
By the $\C B$-reduced twisted cyclic homology we will understand the homology of the chain complex
\[
\begin{CD}
\ldots @>b_\si>> C_n^\la(\C A/\C B,\si) @>b_\si>> C_{n-1}^\la(\C A/\C B,\si)
@>b_\si>> \ldots
\end{CD}
\]
The cycles and boundaries of degree $n$ will be denoted by $Z_n^\la(\C A/\C B,\si)$ and $B_n^\la(\C A/\C B,\si)$ respectively. The homology in degree $n$ will be denoted by $H^\la_n(\C A/\C B,\si)$.
\end{dfn}

\subsection{Generalized trace}\label{s:GenTra}

Let us fix some $m \in \nn$. We let $M_m(\C A)$ denote the unital algebra of $(m\ti m)$ matrices over $\C A$. The automorphism $\si$ extends to an automorphism $\si \in \au\big( M_m(\C A) \big)$ by entry-wise application.
%

Let $n \in \nn \cup \{0\}$. The generalized trace map is defined by
\[
\T{TR} : C_n^\la\big(M_m(\C A),\si\big) \to C_n^\la(\C A,\si)
\, \, \, , \, \, \, \T{TR}(x^0 \olo x^n) := \sum_{(i_0,\ldots,i_n) \in \nn^{n+1}} x^0_{i_0 i_1} \olo x^n_{i_n i_0}
\]
on twisted cyclic chains. We note that $\T{TR}$ is well-defined since we clearly have the identity $\T{TR} \ci \la = \la \ci \T{TR}$.

\begin{lemma}
The generalized trace is a chain map.
\end{lemma}
\begin{proof}
This is a direct computation. See \cite[Corollary 1.2.3]{Lod:CH} and incorporate the automorphism.
\end{proof}

\begin{prop}
The generalized trace induces a chain map
\[
\T{TR} : C^\la_n(M_m(\C A)/M_m(\C B),\si) \to C^\la_n(\C A/\C B,\si)
\]
between reduced twisted cyclic chains.
\end{prop}
\begin{proof}
Let $x^0,\ldots,x^n \in M_m(\C A)$ and let $y \in M_m(\C B)$. We need to verify the identities
\begin{eqnarray}
\T{TR}(y \ot x^1 \olo x^n) &=& 0 \\
\T{TR}(x^0 \cd y \ot x^1 \olo x^n) &=& \T{TR}(x^0 \ot y \cd x^1 \olo x^n)
\end{eqnarray}
for the generalized trace $\T{TR} : C^\la_n(M_m(\C A),\si) \to C^\la_n(\C A/\C B,\si)$. Both of them are straightforward and left to the reader.
%
\end{proof}

\subsection{Reduced cohomology}\label{s:ptc}

\begin{dfn}
By a \emph{$\C B$-reduced twisted cyclic $n$-cochain} we will understand a linear map $\varphi : C_n^\la(\C A/\C B,\si) \to \cc$. The vector space of $\C B$-reduced twisted cyclic $n$-cochains will be denoted by $C^n_\la(\C A/\C B,\si)$.
\end{dfn}

It follows from Proposition \ref{p:thb} that the reduced twisted cyclic cochains can be given the structure of a cochain complex. The coboundary operator $b^\si : C^n_\la(\C A/\C B,\si) \to C^{n+1}_\la(\C A/\C B,\si)$ is given by $b^\si(\varphi) := \varphi \ci b_\si$.

\begin{dfn}
By the \emph{$\C B$-reduced twisted cyclic cohomology} we will understand the homology of the cochain complex
\[
\begin{CD}
\ldots @>b^\si >>C^n_\la(\C A/\C B,\si) @>b^\si>> C^{n+1}_\la(\C A/\C B,\si) @>b^\si>> \ldots
\end{CD}
\]
The cocycles and coboundaries of degree $n$ will be denoted by $Z^n_\la(\C A/\C B,\si)$  and $B^n_\la(\C A/\C B,\si)$ respectively. The homology in degree $n$ will be denoted by $H^n_\la(\C A/\C B,\si)$.
\end{dfn}

\section{The Chern character of a modular Fredholm module}\label{S:ctf}
In this section we will introduce the notion of a modular Fredholm module with respect to a pair consisting of a weight and a trace on a semifinite von Neumann algebra. The concept uses the derived $L^p$-spaces which we constructed in Section \ref{S:der}. We shall then see how to associate a Chern character in twisted cyclic cohomology to a modular Fredholm module. The correct formula is a generalization of the Connes-Chern character see  \cite[IV.1.$\al$]{Con:NCG} to the case where the operator trace is replaced by a weight on a semifinite von Neumann algebra. The construction can be refined to the case where the Fredholm module satisfies an extra invariance property with respect to a subalgebra. In this case the Chern character is a class in reduced twisted cyclic cohomology which we introduced in Section \ref{S:cyc}. The construction of the Chern character relies on the H\"older type inclusions of the derived $L^p$-spaces, see Lemma \ref{l:hol}.

Let $\C N$ be a semifinite von Neumann algebra. Let $\tau : \C N_+ \to [0,\infty]$ be a fixed semifinite normal faithful weight and let $\phi : \C N_+ \to [0,\infty]$ be a semifinite normal faithful weight. We let $\{\si_t\}_{t \in \rr}$ denote the modular group of automorphisms for $\phi$. The notation $\C N^\si$ refers to the fixed point von Neumann algebra for $\{\si_t\}_{t \in \rr}$ (or the centralizer of the weight $\phi$). The notation $\C N_{\T{an}}$ refers to the analytic operators in $\C N$ with respect to $\{\si_t\}_{t \in \rr}$.

\begin{dfn}\label{d:mfm}
Let $p \in [1,\infty)$. By an \emph{odd unital $p$-summable modular Fredholm module} with respect to $(\phi,\tau)$ we will understand the given of a selfadjoint unitary $F \in \C N^\si$ in the centralizer and a unital $*$-subalgebra $\C A \su \C N_{\T{an}}$ such that
\begin{enumerate}
\item The automorphism $\si_z \in \au(\C N_{\T{an}})$ restricts to an automorphism $\si_z \in \au(\C A)$ for all $z \in \cc$.
\item The commutator $[F,x] \in \sw{p}$ lies in the derived $L^p$-space for all $x \in \C A$.
\end{enumerate}
We will use the notation $(\C A,\C N, F)$ for a modular Fredholm module.

We will say that $(\C A,\C N, F)$ is \emph{even} instead of \emph{odd} when there exists a $\zz/2\zz$ grading operator $\ga \in \C N^\si$ such that $x \ga = \ga x$ and $F \ga = - F \ga$ for all $x \in \C A$.
\end{dfn}

Suppose that $\C F := (\C A, \C N, F)$ is a unital $p$-summable modular Fredholm module. Let $n$ denote the smallest integer of the same parity as the Fredholm module such that $n +1 \geq p$.

Define the $(n+1)$-linear map $\ch{n}{\C F} : \C A^{n+1} \to \cc$ by the formula
\[
\ch{n}{\C F}(x_0,\ldots,x_n) := \frac{1}{2}\phi(\ga^{n+1}F[F,x_0][F,x_1]\clc [F,x_n]) \, , \, x_0,\ldots,x_n \in \C A.
\]
We note that $\ch{n}{\C F}$ is well-defined since $[F,x] \in \swa{p} \su \sw{n+1}$ and hence the product $[F,x_0][F,x_1]\clc [F,x_n] \in \sw{1}$ by Lemma \ref{l:inc} and Lemma \ref{l:hol}.

\begin{prop}\label{p:cht}
The character $\ch{n}{\C F}$ is a twisted cyclic cocycle with twist given by the modular automorphism $\si_i \in \au(\C A)$.
\end{prop}
\begin{proof}
Let $x_0,\ldots,x_n \in \C A$. Using the twisted trace property of $\phi : \sw{1} \to \cc$ we get the identity
\[
\begin{split}
\phi(\ga^{n+1}F[F,x_0]\clc [F,x_n])
= (-1)^n \phi(\ga^{n+1}F[F,\si_i(x_n)][F,x_0] \clc [F,x_{n-1}])
\end{split}
\]
This proves that $\ch{n}{\C F} \in C^n_\la(\C A,\si_i)$.

To continue we let $y \in \C A$ and note that
\[
\begin{split}
& \phi(\ga^{n+1}[F,y][F,x_0] \clc [F,x_n]) \\
& \q = -\phi(\ga^{n+1}yF[F,x_0] \clc [F,x_n])
+ (-1)^{n+1} \phi(\ga^{n+1}y[F,x_0] \clc [F,x_n] F)
= 0
\end{split}
\]
where we have used that $\si_{-i}(F) = F$ since $F$ is an element of the centralizer.

Now, to see that $\ch{n}{\C F}$ is a cocycle we compute as follows
\[
\begin{split}
& 2(\ch{n}{\C F} \ci b_\si)(x_0,\ldots,x_n,y) \\
& \q = (-1)^n \phi(\ga^{n+1} F[F,x_0] \clc [F,x_n] \cd y)
+ \phi(\ga^{n+1}F x_0[F,x_1] \clc [F,y]) \\
& \qq + (-1)^{n+1} \phi(\ga^{n+1}F[F,\si_i(y) x_0][F,x_1] \clc [F,x_n]) \\
& \q = (-1)^n \phi(\ga^{n+1} F[F,x_0] \clc [F,x_n] \cd y)
+ \phi(\ga^{n+1}F x_0[F,x_1] \clc [F,y]) \\
& \qq + (-1)^n \phi(\ga^{n+1}[F,\si_i(y)] x_0 F[F,x_1] \clc [F,x_n]) \\
& \qq + (-1)^n \phi(\ga^{n+1}\si_i(y) [F,x_0] F[F,x_1] \clc [F,x_n]) \\
& \q = \phi(\ga^{n+1}[F,x_0] \clc [F,y]) = 0.
\end{split}
\]
This ends the proof of the proposition.
\end{proof}

We will refer to the cohomology class $[\ch{n}{\C F}] \in H^n_\la(\C A,\si_i)$ as the \emph{Chern character} of the modular Fredholm module $\C F$.

It is worthwhile to refine the above situation to the case where $\C B \su \C A$ is a subalgebra of $\C A$ and the modular Fredholm module satisfies certain invariance properties with respect to $\C B$. 
%

Thus, let $\C B \su \C A$ be a subalgebra of $\C A$ such that $\si_z \in \au(\C A)$ restricts to an automorphism of $\C B$ for all $z \in \cc$.

\begin{dfn}
We will say that the modular Fredholm module $\C F$ is \emph{$\C B$-invariant} when the commutator $[F,x]$ is trivial for all $x \in \C B$.
\end{dfn}

We remark that a modular Fredholm module is automatically $\cc$-invariant.

\begin{prop}
Suppose that $\C F$ is a unital $\C B$-invariant $p$-summable modular Fredholm module. Then the twisted cyclic cocycle $\ch{n}{\C F} \in Z^n_\la(\C A,\si_i)$ descends to a $\C B$-reduced twisted cyclic cocycle $\ch{n}{\C F} \in Z^n_\la(\C A/\C B,\si_i)$.
\end{prop}
\begin{proof}
It follows immediately from the $\C B$-invariance that $\ch{n}{\C F}$ is well-defined on the tensor product $(\C A/\C B)^{\ot_\cc (n+1)}$. Furthermore, using that $[F,\cd]$ is a derivation we get that $\ch{n}{\C F}$ descends to $(\C A/\C B)^{\ot_B (n+1)}$. Finally, for any $y \in \C B$ we have the identity
\[
\begin{split}
\phi(\ga^{n+1}F[F,x_0][F,x_1]\ldots [F,x_n y]) & = \phi(\ga^{n+1}\si_i(y)F[F,x_0][F,x_1]\ldots[F,x_n]) \\
& = \phi(\ga^{n+1}F[F,\si_i(y)x_0][F,x_1]\ldots[F,x_n])
\end{split}
\]
which shows that $\ch{n}{\C F}$ is well-defined on $C_n(\C A/\C B)$. This proves the claim of the proposition.
\end{proof}

We will refer to the cohomology class $[\ch{n}{\C F}] \in H^n_\la(\C A/\C B,\si_i)$ as the \emph{$\C B$-reduced Chern character} of the modular Fredholm module.

We end this section by spelling out the standard matrix construction for modular Fredholm modules. Let $k \in \nn$. We can then lift the modular Fredholm module $\C F$ to a modular Fredholm module $\C F_k:= (M_k(\C A), M_k(\C N),F \ot 1_k)$. The new semifinite normal faithful trace and semifinite normal faithful weight are given by $(\tau \ot \T{Tr})(x_{ij}) = \sum_{i=1}^k \tau(x_{ii})$ and $(\phi \ot \T{Tr})(x_{ij}) = \sum_{i=1}^k \phi(x_{ii})$ respectively. In the even case the new grading operator is the diagonal operator $\ga \ot 1_k$. We remark that the new modular group of automorphisms is given by $\{ \si_t \ot 1_k\}_{t \in \rr}$ where $(\si_t \ot 1_k)(x_{ij}) = (\si_t(x_{ij}))$. We leave it to the reader to verify that $\C F_k$ is indeed a modular Fredholm module with respect to $(\phi \ot \T{Tr},\tau \ot \T{Tr})$ when $\C F$ is a modular Fredholm module with respect to $(\phi,\tau)$. The summability and parity is preserved and $\C F_k$ is $M_k(\C B)$-invariant when $\C F$ is $\C B$-invariant. We will often use the notation $\phi,\tau : M_k(\C N)_+ \to [0,\infty]$ and $\{\si_t\}$ for the induced weights, traces and modular automorphism groups on the matrix von Neumann algebras hoping that this will not cause any confusion.

\begin{lemma}\label{l:CheTra}
Suppose that $\C F$ is a unital $\C B$-invariant $p$-summable modular Fredholm module. The $M_k(\C B)$-reduced Chern character $[\ch{n}{\C F_k}] \in H^n_\la(M_k(\C A)/M_k(\C B),\si_i)$ agrees with the cohomology class $\T{TR}^*\big( [\ch{n}{\C F}] \big)$ where $\T{TR}^*: H^n_\la(\C A/\C B,\si_i) \to H^n_\la(M_k(\C A)/M_k(\C B),\si_i)$ is the cohomological version of the generalized trace.
\end{lemma}
\begin{proof}
Let $x^0,\ldots,x^n \in M_k(\C A)$. We then note that
\[
\begin{split}
& \big( (\De \ga^{n+1}F \ot 1_k)[F \ot 1_k,x^0][F \ot 1_k,x^1] \clc [F \ot 1_k,x^n]  \big)_{i_0, i_0} \\
& \q = \sum_{(i_1,\ldots,i_n) \in \nn^n}
\De \ga F[F,x^0_{i_0, i_1}] [F,x^1_{i_1, i_2}] \clc [F,x^n_{i_n,i_0}].
\end{split}
\]
In particular we get that
\[
\begin{split}
& 2 \ch{n}{\C F_k}(x^0,\ldots,x^n) \\
& \q = (\tau \ot \T{Tr})\big( (\De \ga^{n+1}F \ot 1_k)[F \ot 1_k,x^0][F \ot 1_k,x^1]\clc [F \ot 1_k, x^n]\big) \\
& \q = \sum_{(i_0,\ldots,i_n) \in \nn^{n+1}}
\tau(\De \ga^{n+1} F [F,x^0_{i_0, i_1}] [F,x^1_{i_1, i_2}] \clc [F,x^n_{i_n,i_0}]) \\
& \q = (\ch{n}{\C F} \ci \T{TR})(x^0 \olo x^n)
\end{split}
\]
and the lemma is proved.
\end{proof}

\begin{remark}
The definition of a modular Fredholm module and its associated Chern character is very much related to a recent preprint of Rennie, Sitarz and Yamashita, \cite{RSY:TMF}. However, the precise relation between the two concepts of modular Fredholm modules still needs to be clarified.
\end{remark}

\section{Modular spectral triples}\label{S:mss}
In this section we shall see how a finitely summable modular spectral triple gives rise to a modular Fredholm module and thus in particular admits a Chern character. The concept of a modular spectral triple has emerged from the series of papers \cite{CRT:STH,CaPhRe:TIC,CNNR:TEK,ReSe:TLP}. It is a generalization of a semifinite spectral triple, \cite{CPRS:LI1,CPRS:LI2}, which in turn is a generalization of the notion of a spectral triple, \cite{Con:NCG,CoMo:LIF}. The main point in the modular version of a spectral triple is to enable the use of a weight instead of a trace to measure the growth of the resolvent. In order to incorporate some of our examples we need to enhance the definition of a modular spectral triple even further by allowing for twisted commutators in the style of A. Connes and H. Moscovici, \cite{CoMo:TST}.

Let $\C N$ be a semifinite von Neumann algebra. 
%

Let $\phi : \C N_+ \to [0,\infty]$ be a strictly semifinite normal faithful weight with modular group of automorphisms $\{\si_t\}_{t \in \rr}$. The adjective "strictly" means that $\phi$ descends to a semifinite normal faithful trace on the fixed point von Neumann algebra $\C N^\si$. In particular we have the noncommutative $L^p$-spaces $L^p(\C N^\si,\phi)$, $p \in [1,\infty)$. Furthermore, we have the $C^*$-algebra of compacts $\C K(\C N^\si,\phi)$ which is the smallest norm closed $*$-ideal in $\C N^\si$ containing the projections $Q \in \C N^\si$ with $\phi(Q) < \infty$.

\begin{dfn}\label{d:mss}
By an \emph{odd unital modular $\te$-spectral triple} with respect to the weight $\phi : \C N_+ \to [0,\infty]$ we will understand a triple $(\C A,\C N, D)$ such that
\begin{enumerate}
\item $\C A$ is a unital $*$-subalgebra of the semifinite von Neumann algebra $\C N$ which in turn acts on a separable Hilbert space $\C H$.
\item $D : \sD(D) \to \C H$ is an unbounded selfadjoint operator affiliated with the fixed point von Neumann algebra $\C N^\si$.
\item $\te : \C A \to \C A$ is an algebra automorphism of $\C A$.
\item Each element $x \in \C A$ is analytic with respect to the modular automorphism group $\{\si_t\}$ and the operator $\si_z(x) \in \C A$ is contained in $\C A$ for all $z \in \cc$.
\item Each element $x \in \C A$ preserves the domain of $D$ and the twisted commutator $[D,x]_\te = D x - \te(x) D : \sD(D) \to \C H$ extends to an analytic bounded operator.
\item The resolvent $(\la - D)^{-1} \in \C K(\C N^\si,\phi)$ is compact with respect to the trace $\phi : \C N^\si_+ \to [0,\infty]$ on the fixed point von Neumann algebra for all $\la \in \cc\setminus \rr$.
\end{enumerate}

The modular spectral triple is said to be \emph{even} instead of \emph{odd} when there exists a $\zz/(2\zz)$ grading operator $\ga \in \C N^\si$ such that $\ga x = x \ga$ for all $x \in \C A$ and $\ga D = - D \ga$.

Let $p \in [1,\infty)$. We will say that the modular spectral triple is \emph{$p$-summable} when the resolvent $(\la - D)^{-1}$ lies in the $L^p$-space $L^p(\C N^\si,\phi)$ for all $\la \in \cc\setminus \rr$.

We will say that the modular spectral triple is \emph{Lipschitz regular} when the twisted commutator $[|D|,x]_\te$ with the absolute value extends to an analytic bounded operator for all $x \in \C A$.

Let $\C B \su \C A^\si$ be a sub-algebra of the fixed point algebra $\C A^\si := \{x \in \C A \, | \, \si_t(x) = x \, t \in \rr\}$. We will then say that the modular spectral triple is \emph{$\C B$-invariant} when the commutator $[D,x] = 0$ is trivial for all $x \in \C B$.
\end{dfn}

\begin{remark}
It is not hard to see that when $\te = \T{Id}$ we get a unital semifinite spectral triple $(\C A^\si,\C N^\si, D)$ over the fixed point algebra out of a unital modular spectral triple.
\end{remark}

We shall now see how a unital $p$-summable Lipschitz regular modular $\te$-spectral triple gives rise to a unital $p$-summable modular Fredholm module in the sense of Definition \ref{d:mfm}. Let us fix a semifinite normal faithful trace $\tau : \C N_+ \to [0,\infty]$.
%
%
%

\begin{lemma}\label{l:res}
We have the inclusion of $*$-ideals
\[
\sL^p(\C N^\si,\phi) \su \swa{p}
\]
for all $p \in [1,\infty)$.
\end{lemma}
\begin{proof}
Let $x \in \sL^p(\C N^\si,\phi)$. We thus have that $\phi(|x|^p) < \infty$. Now, let $\De : \sD(\De) \to L^2(\tau)$ denote the Radon-Nikodym derivative of $\phi$ with respect to $\tau$. We then have the convergence result
\[
\phi(|x|^p) 
= \lim_{n \to \infty}\tau(\De_{1/n}^{1/2}|x|^p \De_{1/n}^{1/2})
= \lim_{n \to \infty}\tau(\De_{1/n}|x|^p).
\]
It follows that the sequence of bounded positive operators $\De_{1/n}|x|^p$ converges in $L^1(\C N,\tau)$ to some $\tau$-measurable operator $y \in L^1(\C N,\tau)$. We start by proving that $y = \De |x|^p$.

Let $P := E([0,1])$ denote the spectral projection associated with $\De$ and the interval $[0,1]$. We then have that $\ov{Py} = \lim_{n \to \infty}P \De_{1/n}|x|^p = \De P |x|^p$ where the convergence is in measure. On the other hand, $\De^{-1}(1 - P)y = \lim_{n \to \infty} (1 + \De/n)^{-1} (1 - P)|x|^p$ where the convergence is in measure. But by normality we also have that $\lim_{n \to \infty} \tau \big( (1 + \De/n)^{-1} (1 - P)|x|^p\big) = \tau((1 - P)|x|^p)$. These observations imply the identity $\De^{-1}(1 - P)y = (1 - P)|x|^p$ and we conclude that $y = \De |x|^p$, thus $|x|^p \in \sw{1}$.

Next, since $x \in \C N^\si$ we get that $x \in \swa{p}$ if and only if $|x| \in \swa{p}$. We aim at proving the identity $(\De |x|^p)^{1/p} = \De^{1/p} |x|$ which implies the desired result, $|x| \in \swa{p}$.

We start by noting that $(\De |x|^p)^* = \ov{|x|^p \De} \su \De |x|^p$. This implies that $\De |x|^p$ is selfadjoint by $\tau$-measurability. We now define the projection $E(\la,\mu) = E_{\De}(\la) \cd E_{|x|^p}(\mu)$ for all $\la, \mu \in \rr$ where $\{E_{\De}(\la)\}_{\la \in \rr}$ and $\{E_{|x|^p}(\mu)\}_{\mu \in \rr}$ denote the resolutions of the identity associated with $\De$ and $|x|^p$. Remark that $E(\la,\mu)$ is a projection since $|x|^p \in \C N^\si$. It follows that $\{E(\la,\mu)\}_{\la,\mu \in \rr}$ is a resolution of the identity as well. In particular, we have the associated unbounded selfadjoint positive operator
\[
z := \int_{\rr^2} \la \cd \mu \, d E(\la,\mu)
\]
See \cite[Theorem 13.24]{Rud:FA}. Now, it is not hard to see that we have the inclusion $|x|^p \De \su z$ of unbounded operators. It thus follows by selfadjointness that $\De |x|^p = z$. By \cite[Theorem 13.28]{Rud:FA} we then have the identities
\[
z^{1/p} 
= \int_{[0,\infty)^2} \la^{1/p} \cd \mu^{1/p} \, d E(\la,\mu)
= \int_{[0,\infty)^2} \la \cd \mu \, dE(\la^p,\mu^p).
\]
Now, as before we have the inclusion $|x| \De^{1/p} \su \int_{[0,\infty)^2} \la \cd \mu \, dE(\la^p,\mu^p)$ of unbounded operators. (Indeed $E_\De(\la^p) = E_{\De^{1/p}}(\la)$ and $E_{|x|^{1/p}}(\mu) = E_{|x|}(\mu^p)$). But this inclusion implies that $z^{1/p} \su \De^{1/p}|x|$ and it therefore follows by the $\tau$-measurability of $z^{1/p} = (\De |x|^p)^{1/p}$ that $(\De |x|^p)^{1/p} = \De^{1/p} |x|$ as desired.
\end{proof}
%

We let $F = 2P-1$ denote the phase of the unbounded selfadjoint operator $D$. Thus, $P = E_D([0,\infty))$ is the spectral projection associated with $D$ and the halfline $[0,\infty)$.

\begin{prop}\label{p:modfre}
Suppose that $(\C A,\C N,D)$ is a unital modular $\te$-spectral triple with respect to the weight $\phi$. Let $p \in [1,\infty)$ and suppose furthermore that $(\C A,\C N,D)$ is Lipschitz regular and $p$-summable. Then the triple $(\C A,\C N,F)$ is a unital modular $p$-summable Fredholm module with respect to $(\phi,\tau)$. The parity of the two triples is the same and the grading operators coincide in the even case.
\end{prop}
\begin{proof}
We only need to verify the second condition in Definition \ref{d:mfm}. The other conditions follow easily from the assumptions on the modular spectral triple.

Let us form the bounded operator $F_D := D(1 + |D|)^{-1}$ and let $x \in \C A$. We can then compute the commutator as follows,
\[
\begin{split}
[F_D,x] 
& = D(1 + |D|)^{-1} x - x D(1 + |D|)^{-1} \\
& = D(1 + |D|)^{-1} x + [D,\te^{-1}(x)]_\te (1 + |D|)^{-1}
- D \te^{-1}(x) (1 + |D|)^{-1} \\
& = [D,\te^{-1}(x)]_\te (1 + |D|)^{-1} \\
& \q + D(1 + |D|)^{-1}
\big(x (1 + |D|) - (1 + |D|) \te^{-1}(x)\big)(1+ |D|)^{-1} \\
& = [D,\te^{-1}(x)]_\te (1 + |D|)^{-1}
- F_D 
\big( \big[|D|,\te^{-1}(x)\big]_{\te} + [1,\te^{-1}(x)]_{\te} \big) (1 + |D|)^{-1}
\end{split}
\]
This shows that the commutator $[F_D,x]$ lies in the derived $L^p$-space. Indeed, by assumption all the commutators appearing are analytic operators and the resolvent $(1 + |D|)^{-1}$ is derived $p$-summable by Lemma \ref{l:res}. Remark that $L^p(\phi,\tau)$ is a module over $\C N_{\T{an}}$ by Lemma \ref{l:ide}.

To continue we note that $F D = |D|$ where $F = 2P - 1$ is the phase of $D$. In particular we get that $F \cd F_D - 1 = |D|(1 + |D|)^{-1} - 1 = - (1 + |D|)^{-1} \in \sw{p}$. But this implies that $[F,x] \in \sw{p}$ for all $x \in \C A \su \C N_{\T{an}}$ and the proposition is proved.
\end{proof}

We remark that the modular Fredholm module $(\C A,\C N,F)$ is $\C B$-invariant if the modular spectral triple $(\C A,\C N,D)$ is $\C B$-invariant for some subalgebra $\C B \su \C A^\si$.

As a consequence of the last proposition we can make the following definition:

\begin{dfn}
Let $(\C A,\C N,D)$ be a unital $p$-summable Lipschitz regular modular $\te$-spectral triple. By the \emph{Chern character} of $(\C A,\C N,D)$ we will understand the Chern character of the unital $p$-summable modular Fredholm module $(\C A,\C N,F)$ construced in Proposition \ref{p:modfre}. The Chern character will be denoted by $\T{Ch}_\phi^n(\C A, \C N,D) \in H^n_\la(\C A, \si_i)$.
\end{dfn}

In case the modular spectral triple $(\C A,\C N,D)$ is $\C B$-invariant we also have the $\C B$-reduced Chern character $\T{Ch}_\phi^n(\C A,\C N,D) \in H^n_{\la}(\C A/\C B,\si_i)$.

%

\section{The Chern character of a modular unitary}\label{S:cmu}
In order to extract interesting numerical information from an algebra $\C A$ using the Chern character of a modular Fredholm module we need to construct a Chern character from a class of objects in the algebra to the reduced twisted cyclic homology. In the non-twisted odd case a good class to look at is the unitaries in finite matrices over the algebra. In the twisted odd case an extra condition on the relation between the modular automorphism group and the unitary is needed.

Let $\C N$ be a semifinite von Neumann algebra and let $\phi : \C N_+ \to [0,\infty]$ be a semifinite faithful normal weight on $\C N$ with modular group of automorphisms $\{\si_t\}$. Let $\C A \su \C N_{\T{an}}$ be a unital $*$-subalgebra of the analytic operators such that $\si_z(x) \in \C A$ for all $z \in \cc$ and all $x \in \C A$.

We let $\C A^\si$ denote the fixed point algebra for the one-parameter group of automorphisms $\si_t : \C A \to \C A$. Let $\C B \su \C A^\si$ be a unital $*$-subalgebra of the fixed point algebra.

\begin{dfn}
Let $u \in \C U_k(\C A)$ be a unitary over $\C A$. We will say that $u$ is \emph{right $\C B$-modular} when $u^* \si_z(u) \in GL_k(\C B)$ for all $z \in \cc$.
\end{dfn}

\begin{lemma}\label{l:modgrpinv}
Suppose that $u \in \C U_k(\C A)$ is a right $\C B$-modular unitary. The assignment $z \mapsto g_z := u^* \si_z(u) \in GL_k(\C B)$ is then a complex-parameter group of invertibles such that $(g_z)^* = g_{-\ov z}$. In particular we get that the invertible element $g_{it}$ is selfadjoint and positive for all $t \in \rr$.
\end{lemma}
\begin{proof}
Let $z,w \in \cc$. We then have that 
\[
g_z \cd g_w 
= u^*\si_z(u) \cd u^* \si_w(u)
= u^* \si_z(u) \si_z(u^*) \si_{z+w}(u^*)
= g_{z+w}.
\]
Since also $g_0 = 1$ we get that $\{g_z\}_{z \in \cc}$ is a complex-parameter group of invertibles. The assertion on the adjoint follows from the computation
\[
(g_z)^* = (u^*\si_z(u))^* = \si_{\ov z}(u^*)u = u^* \si_{-\ov z}(u) = g_{-\ov z}.
\]
\end{proof}

Let us fix some $n \in \nn$. Suppose that $u \in \C U_k(\C A)$ is a right $\C B$-modular unitary. We then define the $\C B$-reduced $\si_i$-twisted cyclic $(2n-1)$-cycle
\[
\T{Ch}^\phi_{2n-1}(u) := \T{TR}\big( [u] \ot_{\C B} [u^*] \bolo [u] \ot_{\C B} [u^*] \big) \in Z^\la_{2n-1}(\C A / \C B,\si_i).
\]
Here we recall that $\T{TR} : Z^\la_{2n-1}\big(M_k(\C A)/M_k(\C B),\si_i\big) \to Z^\la_{2n-1}(\C A/\C B,\si_i)$ induces the generalized trace on reduced twisted cyclic homology, see Subsection \ref{s:GenTra}.

\begin{dfn}
By the \emph{Chern character} of a right $\C B$-modular unitary in degree $(2n-1)$ we will understand the class $[\T{Ch}^\phi_{2n-1}(u)] \in H^\la_{2n-1}(\C A/\C B,\si_i)$ in $\C B$-reduced twisted cyclic homology.
\end{dfn}

\begin{remark}
The notion of a right $\C B$-modular unitary is related to the first modular $K$-group as defined in \cite[Definition 4.2]{CaPhRe:TIC}. However our condition is not symmetrized. Indeed, the invertible elements $\si_z(u)u^*$ are not required to lie in the subalgebra $M_k(\C B)$. This lack of symmetry might be difficult to capture in the current $K$-theoretic framework.
\end{remark}

\section{The twisted index pairing}\label{S:TwInPa}
Let $\C N$ be a semifinite von Neumann algebra with a fixed semifinite normal faithful trace $\tau : \C N_+ \to [0,\infty]$ and a fixed semifinite normal faithful weight $\phi : \C N_+ \to [0,\infty]$. We let $\{\si_t\}$ denote the modular group of automorphisms for $\phi$. We will assume that $\phi$ is strictly semifinite in the sense that the restriction to the centralizer $\phi : \C N_+^\si \to [0,\infty]$ is a semifinite normal faithful trace.

Let $\C F = (\C A,\C N,F)$ be an odd unital $2n$-summable modular Fredholm module with respect to $(\phi,\tau)$. Furthermore, we let $u \in \C U_k(\C A)$ be some right $\C B$-modular unitary with respect to a unital $*$-subalgebra $\C B \su \C A^\si$ of the fixed point algebra and the modular group of automorphisms $\{\si_t\}$. We assume that $\C F$ is $\C B$-invariant, thus $[F,b] = 0$ for all $b \in \C B$.

By the work carried out in Section \ref{S:ctf} and Section \ref{S:cmu} we can associate a number to the above data. Indeed, we can form the pairing of Chern characters, $\binn{\T{Ch}^\phi_{2n-1}(u),\T{Ch}_\phi^{2n-1}(\C F)} \in \cc$, in $\C B$-reduced twisted cyclic theory.

The purpose of this section is to give an index theoretical interpretation to this number. To be more precise, we shall see that the above pairing agress with a twisted index of the abstract Toeplitz operator $PuP$ where $P$ is the projection onto the eigenspace of $F$ with eigenvalue $1$. The twisting of the index means that the dimension of the kernel and the dimension of the cokernel are measured with respect to two different traces. Let us present the details.

As in the end of Section \ref{S:ctf} we let $\phi := \phi \ot \T{Tr}$ and $\tau := \tau \ot \T{Tr}$ denote the semifinite normal faithful weight and the semifinite normal faithful trace on the matrix von Neumann $M_k(\C N)$. We then have the odd unital $2n$-summable modular Fredholm module $\C F_k:= (M_k(\C A),M_k(\C N), F)$ with respect to $(\phi, \tau)$ where $F := F \ot 1_k$.

Recall from Lemma \ref{l:modgrpinv} that $g_{-i} := u^* \si_{-i}(u) \in GL_k(\C B)$ is a positive invertible matrix which lies in the fixed point von Neumann algebra of $M_k(\C N)$ with respect to the weight $\phi$. In particular we get that the formula $\psi(x) := \phi(g_{-i}^{1/2}x g_{-i}^{1/2})$, $x \in M_k(\C N)_+$ defines a semifinite normal faithful weight on $M_k(\C N)$. Furthermore, we get that the modular group of automorphisms for $\psi$ is given by $\te_t(x) = \si_t(g_t x g_{-t})$ for all $t \in \rr$. Here $g_t := u^* \si_t(u)$. See \cite[Theorem 4.6]{PeTa:RNN}.

We note that the weights $\phi : M_k(\C N)_+ \to [0,\infty]$ and $\psi : M_k(\C N)_+ \to [0,\infty]$ are automatically strictly semifinite since $\phi : \C N_+ \to [0,\infty]$ is strictly semifinite by assumption.

Our first concern is to define a twisted index pairing of $u$ and $\C F$. We start by stating a well-known lemma.

\begin{lemma}\label{l:dif}
Let $x \in \C M$ be an element in a von Neumann algebra $\C M$ and let $v \in \C M$ be the partial isometry of the polar decomposition $x = v|x|$. Let $m \in \nn$. We then have the formulas
\[
\begin{split}
(1 - x^*x)^m - (1 - v^*v) & = (1-x^* x)^m v^*v \q \T{and} \\
(1 - xx^*)^m - (1 - vv^*) & = (1 - x x^*)^m v v^*.
\end{split}
\]
\end{lemma}
%

Next, let $P \in \C M_k(\C N)^{\phi}$ denote the projection $P := (F + 1)/2$.

\begin{lemma}
The orthogonal projection onto the kernel of $P u P + 1 - P$ lies in the trace ideal $L^1(M_k(\C N)^{\psi}, \psi)$ associated with $\psi$ and its centralizer $M_k(\C N)^{\psi}$. The orthogonal projection onto the kernel of $P u^*P + 1 - P$ lies in the trace ideal $L^1(M_k(\C N)^{\phi},\phi)$ associated with $\phi$ and its centralizer $M_k(\C N)^{\phi} = M_k(\C N^\phi)$.
\end{lemma}
\begin{proof}
The projection onto the kernel of $P u P + 1 - P$ lies in $M_k(\C N)^\psi$ since
\[
\begin{split}
\si_t(P u^* P uP + 1 - P)
& = P \si_t(u^*)P \si_t(u) P + 1 - P
= P (u g_t)^* P u g_t P + 1 - P \\
& = g_{-t} (P u^* P  u P  + 1 - P)g_t.
\end{split}
\]
Here we use that $g_t P  = P g_t$ by the $\C B$-invariance of our modular Fredholm module. On the other hand we have that
\[
\si_t(P uP u^*P + 1 - P)  
= Pu g_tPg_{-t} u^*P + 1 - P
= PuP u^*P + 1 - P,
\]
which shows that the projection onto the kernel of $Pu^*P + 1-P$ lies in $M_k(\C N^\phi)$.

Let $v$ be the partial isometry associated with the polar decomposition of the operator $PuP + 1 - P$. We need to show that $1 - v^* v \in L^1(M_k(\C N)^{\psi},\psi)$ and that $1 - vv^* \in L^1(M_k(\C N^\phi) ,\phi)$.

To obtain this, we note that 
\[
\begin{split}
P - PuPu^*P & = P - P[u,P]u^* P \in L^{2n-1}(\phi,\tau) \q \T{and} \\
P - Pu^*PuP & = P - P[u^*,P]u P \in L^{2n-1}(\phi,\tau),
\end{split}
\]
by the $2n$-summability of our modular Fredholm module. Thus, from Lemma \ref{l:hol} we get that $(P - PuPu^* P)^{2n}\, , \, g_i^{1/2}(P - Pu^*PuP)^{2n}g_i^{1/2} \in \sL^1(\phi,\tau)$. It then follows by Lemma \ref{l:weitwt} that
\[
\phi( (P - PuPu^* P)^{2n}) \, , \, \psi((P - Pu^*PuP)^{2n}) < \infty.
\]
We have thus proved the inclusions $(P - PuPu^*P)^{2n} \in L^1(M_k(\C N)^\psi,\psi)$ and $(P - Pu^*PuP)^{2n} \in L^1(M_k(\C N^\phi),\phi)$. But this implies that $1 - v^* v \in L^1(M_k(\C N)^{\psi},\psi)$ and that $1 - vv^* \in L^1(M_k(\C N^\phi) ,\phi_k)$ by an application of Lemma \ref{l:dif}.
\end{proof}

\begin{dfn}
By the \emph{twisted index pairing} of the the right $\C B$-modular unitary $u \in \C U_k(\C A)$ and the $\C B$-invariant odd unital $2n$-summable modular Fredholm module $\C F$ with respect to the pair $(\phi,\tau)$ we will understand the difference
\[
\T{Ind}_\phi(PuP) := \psi( K_{PuP}) - \phi(K_{Pu^*P}) 
\]
where $K_{PuP}$ and $K_{Pu^*P}$ denote the projections onto the kernel of $PuP + 1 - P$ and $Pu^*P + 1 - P$ respectively.
\end{dfn}

We are now going to see that the twisted index pairing coincides with the pairing of Chern characters $\binn{\chh{2n-1}{u},\ch{2n-1}{\C F}} \in \cc$.

\begin{lemma}\label{l:pha}
Let $\C M$ be a von Neumann algebra and let $\{\al_t\}$ be a one-parameter group of $*$-automorphisms. Let $x \in \C M$ and suppose the existence of a positive invertible element $g \in \C M$ such that $\al_t(x) = g^{it} x$ for all $t \in \rr$. We then have the identity $\al_t(v) = g^{it}$, $t \in \rr$ for the application of $\al_t$ to the partial isometry $v$ associated with the polar decomposition of $x = v|x|$.
\end{lemma}
\begin{proof}
We remark that $\al_t(x^*) = \al_t(x)^* = x^* g^{-it}$. In particular, we get that $x^* x \in \C M^\al$ lies in the fixed point von Neumann algebra of $\{\al_t\}$. Now, the partial isometry $v$ is given as the strong limit of the sequence $v_n := x (1/n + |x|)^{-1}$, see \cite[Proposition 2.2.9]{Ped:CAG}. But this shows that $\al_t(v_n) = g^{it} v_n$ converges $\si$-weakly to $\al_t(v)$ for all $t \in \rr$ (since $\al_t$ is normal). This proves the claim of the lemma.
\end{proof}

\begin{prop}\label{p:PaChFr}
Let $u \in \C U_k(\C A)$ be a right $\C B$-modular unitary and let $\C F = (\C A,\C N,F)$ be a $\C B$-invariant odd unital $2n$-summable modular Fredholm module with respect to $(\phi,\tau)$. We then have the identity
\[
\frac{(-1)^{n+1}}{2^{2n-1}}\binn{\T{Ch}^\phi_{2n-1}(u),\T{Ch}_\phi^{2n-1}(\C F)} = \T{Ind}_\phi(PuP) 
\]
between the pairing of Chern characters and the twisted index pairing.
\end{prop}
\begin{proof}
We start by remarking that it follows by the right modular condition on $u$ that $[u] \bolo [u^*] = - [\si_i(u^*)] \bolo [u] = - g_{-i}^{1/2} [u^*] \bolo [u] g_{-i}^{1/2}$ where $g_z := u^* \si_z(u)$. Here we compute inside the reduced twisted cyclic chains $C_{2n-1}^\la(\C A/\C B,\si_i)$. By definition of the Chern classes and by Lemma \ref{l:CheTra} we thus get that
\[
\begin{split}
& \binn{\T{Ch}^\phi_{2n-1}(u),\T{Ch}_\phi^{2n-1}(\C F)}
= \binn{[u] \bolo [u^*],\T{Ch}_{\phi}^{2n-1}(\C F_k)} \\
& \q = \frac{1}{2}\binn{[u] \bolo [u^*],\T{Ch}_{\phi}^{2n-1}(\C F_k)} \\
& \qq - \frac{1}{2}\binn{g_{-i}^{1/2}[u^*] \bolo [u] g_{-i}^{1/2},\T{Ch}_{\phi}^{2n-1}(\C F_k)} \\
& \q = \frac{1}{4}\phi\big( F[F,u] \clc [F,u^*]\big)
- \frac{1}{4}\phi\big( g_{-i}^{1/2}F[F,u^*] \clc [F,u] g_{-i}^{1/2}\big) \\
& \q = 2^{2n-1} \phi \big( P[P,u] \clc [P,u^*]\big)
- 2^{2n-1}\phi\big( g_{-i}^{1/2}P[P,u^*] \clc [P,u] g_{-i}^{1/2}\big).
\end{split}
\]
To continue we note that $P[P,u][P,u^*]P = PuPu^*P- P$ and that $P[P,u^*][P,u]P = Pu^*PuP - P$. To ease the notation we let $PuP + 1-P = x$. We thus have that
\[
\binn{\T{Ch}^\phi_{2n-1}(u),\T{Ch}^{2n-1}_\phi(\C F)}
= 2^{2n-1}\phi((x x^* - 1)^n) - 2^{2n-1} \phi(g_{-i}^{1/2}(x^*x - 1)^n g_{-i}^{1/2}).
\]
Let $v$ denote the partial isometry associated with the polar decomposition $x = v|x|$. By an application of Lemma \ref{l:dif} we then get the identity
\[
\begin{split}
& \frac{(-1)^{n+1}}{2^{2n-1}}\binn{\T{Ch}^\phi_{2n-1}(u),\T{Ch}^{2n-1}_\phi(\C F)} 
- \T{Ind}_\phi(PuP) \\
& \q = \phi\big( g_{-i}^{1/2} (1 - x^*x)^n g_{-i}^{1/2} \big)
 - \phi(1 - x x^*)^n - \phi\big( g_{-i}^{1/2}(1 - v^*v)g_{-i}^{1/2}\big) 
+ \phi(1 - vv^*) \\
& \q = \phi\big(g_{-i}^{1/2}(1 - x^*x)^n v^*v g_{-i}^{1/2}\big) - \phi((1- xx^*)^n vv^*).
\end{split}
\]
We now remark that $(1 - x^*x)v^* = v^*(1 - xx^*)$. Thus, Proposition \ref{p:twt} and Lemma \ref{l:pha} we get that
\[
\begin{split}
& \phi\big(g_{-i}(1 - x^*x)^n v^*v\big) 
- \phi((1- xx^*)^n vv^*)
= \phi\big(g_{-i}v^* (1- xx^*)^n v \big)
- \phi((1- xx^*)^n vv^*) \\
& \q = \phi\big(g_{-i}v^* (1- xx^*)^n v \big)
- \phi(\si_i(v^*)(1- xx^*)^n v)
= 0,
\end{split}
\]
since $\si_z(P x^* P) = g_{-z} P x^* P$.
\end{proof}

\begin{remark}
It is an interesting question to compare our twisted index pairing with the index problem appearing in \cite[Section 2.2]{ReSe:TLP}. We leave this as a subject for further research.
\end{remark}

\section{Example 1: KMS-states}\label{s:kms}
In recent work of Carey, Neshveyev, Nest and Rennie a modular index pairing is constructed, \cite{CNNR:TEK}. The main input data is a $C^*$-algebra $A$ with a strongly continuous action of the circle $\si : S^1 \to \T{Aut}(A)$. The $C^*$-algebra comes equipped with a fixed KMS-state $\phi : A \to \cc$ where the KMS-condition is with respect to the circle action at $\be \in (0,\infty)$. In this section we shall outline how their constructions give rise to an $A^\si$-invariant $1$-summable modular spectral triple $(\C A,\C N,D)$ in the sense of Section \ref{S:mss}. Here $A^\si$ is the fixed point algebra for the circle action. In particular we get a Chern character $\T{Ch}_{\phi_D}^1(\C A,\C H,D) \in H^1_\la(\C A/A^\si, \si_{-i\be})$ with values in the $A^\si$-reduced twisted cyclic cohomology. The twist is given by the analytic extension of the circle action to the value $-i \be \in \cc$. It is then of interest to compare the twisted index pairing with the right modular unitaries as constructed in Section \ref{S:cmu} with the modular index map as it appears in \cite[Theorem 4.11]{CNNR:TEK}.

Let us briefly recall the main constructions of \cite{CNNR:TEK}. 

We start with a $C^*$-algebra $A$ and a strongly continuous actions of the circle $\si : S^1 \to \T{Aut}(A)$. We will think of $\si$ as a $2\pi$-periodic one parameter group of automorphisms $\si_t := \si(e^{it})$. The notation $A^\si := \{x \in A \,|\, \si_t(x) = x \, , \, \forall t \in \rr\}$ refers to the associated fixed point algebra. Under an extra assumption called the spectral subspace assumption, see \cite[Definition 2.2]{CNNR:TEK}, an unbounded $A$-$A^\si$ Kasparov module $(\C A, X, D_X)$ is found, see \cite[Proposition 2.9]{CNNR:TEK}. The dense $*$-subalgebra $\C A \su A$ is the span of the eigenspaces for the circle action, the Hilbert $C^*$-module $X$ is a completion of the algebra $A$ and the selfadjoint regular operator $D_X : \sD(D_X) \to X$ is the generator of the circle action. The unbounded operator $D_X$ satisfies the invariance condition $[D_X,x] = 0$ for all $x \in A^\si$.

In order to obtain numerical invariants the extra data of a faithful KMS-state $\phi : A \to \cc$ is introduced. The KMS-condition is understood to be with respect to the circle action at some $\be \in (0,\infty)$, thus we have the identity $\phi(y\si_{z + i\be}(x)) = \phi(\si_z(x)y)$ for all $z \in \cc$, all $y \in A$ and all analytic elements $x \in A_{\T{an}}$.

The KMS-condition implies that the restriction to the fixed point algebra $\tau:= \phi|_{A^\si} : A^\si \to \cc$ is a faithful tracial state. We thus have the Hilbert space $\C H$ obtained as the completion of $X$ with respect to the inner product $\inn{\xi,\eta}_\tau := \tau(\inn{\xi,\eta}_{A^\si})$. We let $\La : X \to \C H$ denote the inclusion. Note that $\C H$ is a right module over the fixed point algebra $A^\si$. We let $\C N$ denote the von Neumann algebra of $A^\si$-module maps. Thus $\C N$ is defined by
\[
\C N := \big\{ T \in \C L(\C H) \, | \, T(\xi \cd x) = T(\xi) \cd x \, , \, \forall \xi \in \C H \, , \, \forall x \in A^\si \big\}.
\]
%
The von Neumann algebra $\C N$ is a semifinite von Neumann algebra with semifinite normal faithful trace $\T{Tr}_\tau : \C N_+ \to [0,\infty]$ defined by
\[
\T{Tr}_\tau(T) := \T{sup}_{I \su X} \sum_{\xi \in I} \inn{ \La(\xi),T\La(\xi)}_\tau
\]
where the supremum is taken over all finite subsets $I \su X$ such that $\sum_{\xi \in I} \Te_{\xi,\xi} \leq 1$. Here $\Te_{\xi,\xi} \in \C K(X)$ is the compact operator given by $\Te_{\xi,\xi}(\ze) = \xi \inn{\xi,\ze}_{A^\si}$. See also \cite[Theorem 1.1]{LacNes:KQP}.

The unbounded selfadjoint operator $D_X : \sD(D_X) \to X$ descends to an unbounded selfadjoint operator $D$ on $\C H$ by localization which is affiliated with $\C N$. The problem is now that the resolvent $x(i + D)^{-1}$, $x \in A$ is not usually compact with respect to the trace $\T{Tr}_\tau : \C N_+ \to [0,\infty]$. However, letting $\De = e^{-\be D}$ and passing to the strictly semifinite normal faithful weight $\phi_D : \C N_+ \to [0,\infty]$, $\phi_D(T) := \lim_{n \to \infty} \T{Tr}_\tau(\De_{1/n}^{1/2} T \De_{1/n}^{1/2})$ we have the following:

\begin{prop}\label{p:msskms}
Suppose that $A$ is unital. The triple $(\C A,\C N,D)$ is a unital odd $p$-summable $A^\si$-invariant Lipschitz regular modular spectral triple with respect to $\phi_D$ in the sense of Definition \ref{d:mss} for all $p > 1$.
\end{prop}
\begin{proof}
The proof follows from the above observations and \cite[Lemma 4.8]{CNNR:TEK}. Note that the modular group of automorphisms for $\phi_D$ is given by $\si^{\phi_D}_t (T) = e^{-\be i t D} T e^{\be i t D}$. In particular we have that $\si^{\phi_D}_t(x) = \si_{-\be t}(x)$ for all $x \in A$. Remark also that the commutator estimates can be proved as for the Dirac operator on the circle.
\end{proof}

As a consequence of Proposition \ref{p:msskms} and the results of Section \ref{S:mss} and Section \ref{S:ctf} we have the Chern character $\T{Ch}_{\phi_D}^1(\C A, \C N,D) \in H^1_\la(\C A/ A^\si, \si_{- i \be})$ in $A^\si$-reduced twisted cyclic cohomology. Furthermore, as we saw in Proposition \ref{p:PaChFr} the pairing of Chern characters $\binn{\T{Ch}^{\phi_D}_1(u),\T{Ch}^1_{\phi_D}(\C A,\C N,D)} \in \cc$ computes the twisted index of the abstract Toeplitz operator $PuP$ for any right $A^\si$-modular unitary $u \in \C U_k(\C A)$. Here $P = E_D([0,\infty))$ denotes the spectral projection associated with the Dirac operator $D$ and the halfline $[0,\infty)$.

\section{Example 2: Quantum SU(2)}
In this section we present a complete computation of the twisted index pairing for the modular spectral triple $(\C A(SU_q(2)),\C H,D_q)$ over quantum $SU(2)$ which was introduced in \cite{KaaSen:TSQ}. To be more precise, we shall see that the twisted index pairing $\T{Ind}_{\phi}(PuP)$ is given by the formula
\[
\T{Ind}_\phi(PuP) 
= \frac{1}{2}\binn{\T{Ch}^\phi_1(u),\T{Ch}_\phi^1(\C A(SU_q(2)),\C H,D_q)}
= C/2 \cd \T{Tr}\big( u \cd u^* - \si_i(u^*) \cd u\big)
\]
for any right modular unitary $u \in \C U_k(\C A(SU_q(2)))$. The constant $C > 0$ is explicit and given by $C = \frac{q^{1/2}}{(1 - q^{1/2})^2} + \frac{q^{3/2}}{(1 - q^{3/2})^2}$. 
%

In particular we get the formula
\[
\T{Ind}_\phi(Pu^lP) = C/2 \cd \big( (2l+1) - [2l+1]_{q^{1/2}} \big)
\q l \in \frac{1}{2}\nn \cup \{0\}
\]
for the evaluation at the sequence of corepresentation unitaries. It is
worthwhile to notice that all these values are symmetric polynomials in $q$
and $q^{-1}$ with integer coefficients up to the constant $C/2$.

\subsection{Preliminaries on quantum SU(2)}
We start by fixing some notation and conventions for the quantum group versions of the classical Lie-group $SU(2)$. These $q$-deformations of the special unitary matrices were introduced by Woronowicz in \cite{Wor:TDC}, and we denote them by $SU_q(2)$. Our main reference for this part is the book \cite{KlSc:QGR} by Klimyk and Schm\"udgen.

Let $q \in (0,1)$. We let $\C A:= \C A(SU_q(2))$ denote the coordinate algebra for the quantum group $SU_q(2)$. This is the unital $*$-algebra generated by the symbols $a,b,c,d \in \C A$ with algebra relations
\begin{equation}\label{eq:gen}
\begin{split}
& ab = qba \, \, , \, \, ac = q ca 
\, \, , \, \,  bd = q db 
\, \, , \, \,  cd = q dc \, \, , \, \, bc = cb \\
& ad - q bc = 1 \, \, , \, \,  da - q^{-1} bc = 1.
\end{split}
\end{equation}
and involution $* : \C A \to \C A$ given by $a^* = d$ and $b^* = -q c$. The coordinate $*$-algebra $\C A$ can be equipped with the extra structure of a Hopf $*$-algebra, see \cite[Section 4.1.2]{KlSc:QGR} for example.
%
%

We let $\C U$ be the unital $*$-algebra generated by the symbols $e,f,k,k^{-1} \in \C U$ with algebra relations
\begin{equation}\label{eq:relU}
\begin{split}
& k^{-1}k = 1 = kk^{-1} \\
& ke = q ek \, \, , \, \,  kf = q^{-1}fk \\
& ef - fe = (k^2 - k^{-2})/ (q - q^{-1})
\end{split}
\end{equation}
and involution $* : \C U \to \C U$ given by $e^* = f$ and $k^* = k$. The $*$-algebra $\C U$ can also be given the structure of a Hopf $*$-algebra as in \cite[Section 3.1.2]{KlSc:QGR}.

The quantum group $SU_q(2)$ comes equipped with a Haar-state which we will denote by $h : SU_q(2) \to \cc$. The Haar-state is a faithful state \cite{Wor:TDC}. We will use the notation $\C H_h$ for the associated $GNS$-space which then comes equipped with an action $\pi : SU_q(2) \to \C L(\C H_h)$ of $SU_q(2)$ as left-multiplication operators.


The coordinate algebra $\C A$ has a vector space basis $\{t^l_{mn} \in \C A\,|\, 2l \in \nn \cup \{0\} \, , \, n,m = -l,\ldots,l\}$, which are related to the Peter-Weyl decomposition of the Hopf $*$-algebra $\C A(SU_q(2))$, see \cite[Theorem 17, Section 4.3.2]{KlSc:QGR}. This basis for $\C A$ is actually an orthogonal basis for the surrounding Hilbert space $\C H_h$. We will use the notation $\{\xi^l_{mn}\}$ for the associated orthonormal basis for $\C H_h$ and refer to this basis as the \emph{corepresentation basis}. Each of the matrices $u^l := \{t^l_{mn}\}_{m,n = -l,\ldots,l}$ is a unitary and we will refer to them as the \emph{corepresentation unitaries}.

There is a $*$-representation of $\C U$ as unbounded operators on $\C H_h$. The domain of each unbounded operator is the $*$-algebra $\C A \su \C H_h$ regarded as a subspace of the GNS-space $\C H_h$. The corresponding adjoint operation is the \emph{formal adjoint} of unbounded operators. The unbounded operators $E,F,K$ and $K^{-1}$ associated with the generators of $\C U$ can be neatly described in terms of the corepresentation basis. Indeed we have that
\begin{equation}\label{eq:expefk}
\begin{split}
& E : \xi^l_{mn} \mapsto \sqrt{[l - n]_q[l+n+1]_q}\, \xi^l_{m,n+1} \\
& F : \xi^l_{mn} \mapsto \sqrt{[l+n]_q[l-n+1]_q} \, \xi^l_{m,n-1} \\
& K : \xi^l_{mn} \mapsto q^n \xi^l_{mn}.
\end{split}
\end{equation}
Here we use the notation $[z]_q = \frac{q^z - q^{-z}}{q - q^{-1}}$ for the \emph{$q$-integer} associated with some complex number $z \in \cc$. We remark that the $*$-representation comes from the left action associated with the dual pairing of Hopf $*$-algebras $\C U$ and $\C A(SU_q(2))$, see \cite[Theorem 21]{KlSc:QGR}.

We end by describing the multiplication operators $\pi(a), \pi(c)$ in terms of the corepresentation basis. We have that
\[
\begin{split}
& \pi(a) : \xi^l_{mn} \mapsto \al^+_{lmn} \xi^{l+1/2}_{m-1/2,n-1/2} + \al^-_{lmn} \xi^{l-1/2}_{m-1/2,n-1/2} \\
& \pi(c) : \xi^l_{mn} \mapsto \ga^+_{lmn} \xi^{l+1/2}_{m+1/2,n-1/2} + \ga^-_{lmn} \xi^{l-1/2}_{m+1/2,n-1/2},
\end{split}
\]
where the coefficients are given by
\begin{equation}\label{eq:cof}
\begin{split}
\al^+_{lmn} & = q^{(2l+m+n+1)/2}
\left(
\frac{[l-m+1]_q[l-n+1]_q}{[2l+1]_q[2l+2]_q}
\right)^{1/2} \\
\al^-_{lmn} & = q^{(-2l + m+n-1)/2} 
\left(
\frac{[l+m]_q[l+n]_q}{[2l]_q[2l+1]_q}
\right)^{1/2} \\
\ga^+_{lmn} & = q^{(m+n-1)/2} \left( \frac{[l+m+1]_q[l-n+1]_q}{[2l+1]_q[2l+2]_q} \right)^{1/2} \\
\ga^-_{lmn} & = - q^{(m+n-1)/2} \left( \frac{[l-m]_q}{[l+n]_q}{[2l]_q [2l+1]_q} \right)^{1/2}.
\end{split}
\end{equation}
See \cite[Proposition 3.3]{DabLanSitSuiVar:Dos}. Note however that our conventions are different from the conventions applied by D\c{a}browski \emph{et al.}

\subsection{The modular spectral triple}
Let $\C H = \C H_h \op \C H_h$ denote the separable Hilbert space given by two copies of the GNS-space for the Haar-state. We will use the notation $\pi := \pi \ot 1_2 : \C A \to \C L(\C H)$ for the diagonal representation with the GNS-representation on the diagonal.

We have the unbounded symmetric operator
\[
D_{q,0} : \C A \op \C A \to \C H \q D_{q,0} := \ma{cc}{(q^{-1}K^{-2} - 1)/(q - q^{-1}) & FK^{-1} q^{1/2} \\
EK^{-1} q^{-1/2} & (1 - q K^{-2})/(q - q^{-1}) } 
\]
and let $D_q := \ov{D_{q,0}}$ denote the closure. We will refer to $D_q$ as the \emph{$q$-Dirac operator}.

Furthermore, we have the unbounded diagonal operator
\[
\De_0 : \C A \op \C A \to \C H \q \De_0 := \ma{cc}{K^{-1} q^{-1/2} & 0 \\ 0 & K^{-1} q^{1/2}}
\]
and let $\De := \ov{\De_0}$ denote the closure. Then $\De$ is an unbounded positive selfadjoint injective operator and we thus have the associated weight $\phi : \C L(\C H)_+ \to [0,\infty]$ defined by $\phi(T) := \lim_{n \to \infty}\T{Tr}(\De_{1/n}^{1/2} T \De_{1/n}^{1/2})$. The weight $\phi$ is strictly semifinite since $\De$ is diagonal. The restriction of the modular group of automorphisms $\{\si_t\}$ to the algebra $\C A(SU_q(2))$ is given by
\[
\si_t(\xi^l_{mn}) = q^{-itn} \xi^l_{mn} \q l \in \frac{1}{2}\nn \cup \{0\} \, , \, m,n \in \{-l,\ldots,l\}
\]
in terms of the corepresentation basis.

The main results of the paper \cite{KaaSen:TSQ} can now be summarized as follows:

\begin{theorem}\label{t:msssuq}
The triple $\C D_q := (\C A,\C H,D_q)$ is an odd unital $1$-summable Lipschitz regular modular $\si_{-2i}$-spectral triple with respect to the weight $\phi : \C L(\C H)_+ \to [0,\infty]$ (and the operator trace).
\end{theorem}

As a consequence of Theorem \ref{t:msssuq} we get the reduced Chern character
\[
\T{Ch}_\phi^1(\C A , \C H, D_q) \in H_\la^1(\C A/\cc, \si_i)
\]
in reduced twisted cyclic homology. The main achievement of the next subsections is a complete computation of this reduced Chern character.

\begin{remark}
We note that the triple $(\C A,\C H,D_q)$ is related to but distinct from the triple $(\C A,\C H,\C D)$ constructed and analyzed in \cite{KRS:RFS}. The difference of the Dirac operators is an unbounded diagonal operator when restricted to the dense subspace $\C A \op \C A \su  \C H$.
\end{remark}

\subsection{Approximation of the generators}
In order to compute the reduced Chern character of our modular spectral triple it is convenient to work with a different representation of $SU_q(2)$. This alternative representation approximates the GNS-representation in the sense that the difference of representations $\rho - \pi$ takes values in the derived $L^1$-space. The technique is related to the approximation procedure of \cite[Lemma 4.4]{DALW:DPQ} and \cite[Lemma 3]{KrEl:RFP}.

Let us define the bounded operator
\begin{equation}\label{eq:rep}
\rho(x) : \xi^l_{mn} \mapsto \fork{ccc}{
\sqrt{1 - q^{2l+2m}}\xi^{l-}_{m-,n-} & & x = a \\
- q^{l+m+1} \xi^{l+}_{m-,n+} & & x = b \\
q^{l+m} \xi^{l-}_{m+,n-} & & x = c \\
\sqrt{1 - q^{2l+2m+2}}\xi^{l+}_{m+,n+} & & x = d
}
\end{equation}
on $\C H_h$ for each generator of $SU_q(2)$. Here we use the short notation $k+ = k +1/2$ and $k- = k-1/2$ for addition and subtraction of a half.

\begin{prop}\label{p:rep}
The assignment in \eqref{eq:rep} defines a representation $\rho : SU_q(2) \to \C L(\C H_h)$ of quantum $SU(2)$.
\end{prop}
\begin{proof}
Since $SU_q(2)$ is a universal $C^*$-algebra we only need to verify that the bounded operators defined by \eqref{eq:rep} satisfy the $*$-algebra relations of the generators $a,b,c,d \in SU_q(2)$, see \eqref{eq:gen}. This can be proved by a direct computation.
\end{proof}

\begin{prop}\label{p:appro}
The difference $\pi(x) - \rho(x) \in L^1(\phi)$ lies in the first derived Schatten ideal associated with the weight $\phi : \C L(\C H_h)_+ \to [0,\infty]$ (and the operator trace) for all $x \in \C A$.
\end{prop}
\begin{proof}
We only need to prove that the unbounded operator $(\pi(x) - \rho(x)) K^{-1} : \C A \to \C H_h$ extends to a bounded operator of trace class for $x = a$ or $x = c$.

Let us assume that $x = a$. We recall from \eqref{eq:cof} that $\pi(a) = a_+ + a_-  \in \C L(\C H_h)$ where the terms are defined by $a_+(\xi^l_{mn}) = \al^+_{lmn} \xi^{l+}_{l-,m-}$ and $a_-(\xi^l_{mn}) = \al^-_{lmn} \xi^{l-}_{l-,m-}$. The coefficients are given by
\[
\begin{split}
\al^+_{lmn} & = q^{(2l+m+n+1)/2}
\left(
\frac{[l-m+1]_q[l-n+1]_q}{[2l+1]_q[2l+2]_q}
\right)^{1/2} \\
\al^-_{lmn} & = q^{(-2l + m+n-1)/2} 
\left(
\frac{[l+m]_q[l+n]_q}{[2l]_q[2l+1]_q}
\right)^{1/2}.
\end{split}
\]

We start by showing that $a_+ \in L^1(\phi)$. This is equivalent to proving that the sum
\[
\T{Tr}( |a_+ K^{-1}|)
= \sum_{l= 0,1/2,\ldots} \sum_{n,m = -l}^l \al^+_{lmn} q^{-n} < \infty
\]
is convergent. However, it follows from basic computations that $\al^+_{lmn} q^{-n} \leq C_1 q^l$ for some constant $C_1 > 0$ which is independent of the indices. This proves that $a_+ \in L^1(\phi)$ since the sum $\sum_{l=0,1/2,\ldots} q^l \cd (2l+1)^2 < \infty$ is convergent.

We continue by proving that $a_- - \rho(a) \in L^1(\phi)$. This is equivalent to proving that the sum
\[
\T{Tr}|(a_- - \rho(a))K^{-1}| = \sum_{l=0,1/2,\ldots} \sum_{m,n = -l+1}^l
|\al^-_{lmn} - \sqrt{1 - q^{2l + 2m}}| q^{-n} < \infty
\]
is finite. Remark that $\xi^{l-}_{m-,n-}= 0$ whenever $m=-l$ or $n=-l$. In order to estimate this sum we note that there exists a constant $C_2 > 0$ such that $|(\al^-_{lmn})^2 - (1 - q^{2l+2m})|q^{-n} \leq C_2 \cd q^l$ for all indices $l,m,n$. This can be proved by straightforward computations. Furthermore we clearly have that $\sqrt{1 -q^2} \leq \al^-_{lmn} + \sqrt{1 - q^{2l+2m}}$ for all $l = 0,1/2,1,\ldots$ and $n,m = -l+1,\ldots,l$. These observations entail that
\[
\begin{split}
|\al^-_{lmn} - \sqrt{1 - q^{2l + 2m}}| q^{-n}
& = \frac{|(\al^-_{lmn})^2 - (1 - q^{2l+2m})|}
{\al^-_{lmn} + \sqrt{1 - q^{2l+2m}}}q^{-n} \\
& \leq C_2 \cd q^l \cd (1 - q^2)^{-1/2}
\end{split}
\]
for all $l = 0,1/2,1,\ldots$, $n,m = -l+1,\ldots,l$. In particular we get that
\[
\T{Tr}|(a_- - \rho(a))K^{-1}| \leq C_2 (1 - q^2)^{-1/2} 
\sum_{l=0,1/2,\ldots} q^l (2l)^2 < \infty.
\]
We have thus proved that $\pi(a) - \rho(a) \in L^1(\phi)$. The corresponding result for $x = c$ can be proved by similar methods and is left to the reader.
\end{proof}

\subsection{Computation of the Chern character}\label{s:CoChCh}
Let $F = D_q |D_q|^{-1}$ denote the phase of the $q$-Dirac operator $D_q$. We recall from Section \ref{S:ctf} and Section \ref{S:mss} that the formula
\begin{equation}\label{eq:reccyccoc}
\ch{1}{\C D_q} := \T{Ch}_\phi^1(\C A, \C H, D_q) : (x,y) \mapsto \frac{1}{2}\T{Tr}(\De F[F,\pi(x)][F,\pi(y)])
\end{equation}
defines a reduced twisted cyclic $1$-cocycle on the coordinate algebra $\C A$ for quantum $SU(2)$. The purpose of this section is to show that the cocycle $\ch{1}{\C D_q}$ agrees with a "local" cocycle up to reduced twisted coboundaries. To be more precise, we shall see that the reduced Chern character $\ch{1}{\C D_q} \in H_\la^1(\C A/\cc, \si_i)$ in reduced twisted cyclic cohomology is represented by the reduced twisted cyclic cocycle
\[
\La : (x,y) \mapsto C \cd h( b_{\si_i}(x,y)).
\]
Here $h : SU_q(2) \to \cc$ is the Haar-state and $C > 0$ is an explicit constant. In order to prove this result, we shall make use of the "approximate" representation for $SU_q(2)$ which we found in Proposition \ref{p:rep} and Proposition \ref{p:appro}.

We let $\al \in Z_\la^1(\C A/\cc,\si)$ denote the reduced twisted cyclic cocycle defined by the formula
\[
\al : (x,y) 
\mapsto \frac{1}{2}\T{Tr}(\De F[F,\rho(x)][F,\rho(y)]).
\]
Thus, comparing with \eqref{eq:reccyccoc}, we have simply replaced the usual representation of $SU_q(2)$ by the representation of Proposition \ref{p:rep}.

\begin{lemma}\label{l:crep}
We have the identity $\al = \T{Ch}_\phi^1(\C D_q)$ in the cohomology group $H_\la^1(\C A/\cc,\si_i)$.
\end{lemma}
\begin{proof}
We define a reduced twisted cyclic $0$-cochain by the formula $\Gamma : x \mapsto \T{Tr}(\De (\pi(x) - \rho(x))F)$. It is then not hard to verify that $b^{\si_i}(\Gamma) = \al -\T{Ch}_\phi^1(\C D_q)$ which in turn proves the lemma.
\end{proof}

Let us introduce the positive trace class operator $T : \xi^l_{mn} \mapsto q^{l+1/2} \xi^l_{mn}$. We note that the unbounded operator $T \De^{-1} : \T{Dom}(\De^{-1}) \to \C H$ extends to a bounded operator and that we have the identity $|D_q|^{-1} = (q^{-1} - q)(1 - T^2)^{-1} (T \De^{-1})$, see \cite[Lemma 3.3]{KaaSen:TSQ}.

\begin{prop}\label{p:comp}
The reduced twisted cyclic cocycle $\al \in Z_\la^1(\C A/\cc,\si_i)$ is given by the formula
\[
\al(x,y) = (q^{-1} - q)\T{Tr}\big(\rho(b_\si(x,y))R\big).
\]
for all $x,y \in \C A/\cc$. Here $R = (T^2 K^{-2}) (1 - T^2)^{-1} T$ is a positive trace class operator.
\end{prop}
\begin{proof}
Let $x,y \in \C A$ and let $\xi \in \C A \op \C A$. We have that
\[
\begin{split}
[F,\rho(y)](\xi)
& = (q^{-1} - q)\big[D_q(1 - T^2)^{-1} (T \De^{-1}),\rho(y)\big](\xi) \\
& = (q^{-1} - q)\big(D_q (1- T^2)^{-1} T \rho(\si_i(y))
- \rho(y) D_q (1 - T^2)^{-1} T\big) \De^{-1}(\xi).
\end{split}
\]

Now, since the diagonal of $D_{q,0} : \C A \op \C A \to \C H_h \op \C H_h$ consists of the unbounded operators $(1 - q^{-1} K^{-2})/(q^{-1} - q)$ and $(q K^{-2} - 1)/(q^{-1} - q)$ we get that
\begin{equation}\label{eq:second}
\begin{split}
& \al(x,y) = \T{Tr}(\De \rho(x)[F,\rho(y)]) \\
& \q = \T{Tr}\Big( \rho(x) 
\big( (1 -q^{-1} K^{-2}) + (q K^{-2} - 1) \big) (1 - T^2)^{-1} T \rho(\si_i(y)) \\
& \qqq - \rho(x y) \big( (1 - q^{-1} K^{-2}) + (q K^{-2} - 1) \big) (1 - T^2)^{-1} T \Big) \\
& \q = (q^{-1} - q)
\T{Tr}\big(\rho(xy) K^{-2} (1- T^2)^{-1}T 
- \rho(x) K^{-2}(1 - T^2)^{-1} T \rho(\si_i(y)) 
\big).
\end{split}
\end{equation}
To continue from this point we observe that
\[
K^{-2}(1 - T^2)^{-1} T(\xi)
= K^{-2} T^2 (1 - T^2)^{-1} T(\xi) + K^{-2} T(\xi)
= R(\xi) + K^{-2} T(\xi)
\]
for all $\xi \in \C A \su \C H_h$, where we recall that $R = (K^{-2} T^2)(1 - T^2)^{-1} T$ is an operator of trace class. In particular we get that
\begin{equation}\label{eq:last}
\begin{split}
& \T{Tr}\big(\rho(xy) K^{-2} (1- T^2)^{-1}T 
- \rho(x) K^{-2}(1 - T^2)^{-1} T \rho(\si_i(y)) 
\big) \\
& \q = 
\T{Tr}\big( \rho(xy) K^{-2} T - \rho(x) K^{-2} T \rho(\si_i(y)) \big)
+ \T{Tr}\big( \rho(xy) R - \rho(x) R \rho(\si_i(y)) \big) \\
& \q = 
\T{Tr}\big( \rho(b_{\si_i}(x,y)) R \big).
\end{split}
\end{equation}
Here we have used the identity $K^{-2} T \rho(z)(\xi) = \rho(\si_{-i}(z)) K^{-2} T(\xi)$ which is valid for all $z \in \C A$ and all $\xi \in \C A \su \C H_h$. The result of the lemma now follows by a combination of \eqref{eq:second} and \eqref{eq:last}.
\end{proof}

\begin{lemma}\label{l:inva}
The functional $x \mapsto \T{Tr}(\rho(x) \cd R)$ is invariant under the automorphism $\si_i \in \au(\C A)$. Thus we have that $\T{Tr}(\rho(\si_i(x)) R) = \T{Tr}(\rho(x) R)$ for all $x \in \C A$. The value at $1$ is given by 
\[
\T{Tr}(R) = \frac{1}{q^{-1} - q}\Big(\frac{q^{1/2}}{(1 - q^{1/2})^2} + \frac{q^{3/2}}{(1 - q^{3/2})^2} \Big).
\]
\end{lemma}
\begin{proof}
For each $z \in \cc$ we define the operator $T^z : \sD(T^z) \to \C H_h$ as the closure of the diagonal operator given by $\xi^l_{mn} \mapsto q^{(l+1/2)z} \xi^l_{mn}$. Let $S \in \C L(\C H_h)$. It is then not hard to see that the zeta-function
\[
z \mapsto \ze_z(S) = \T{Tr}(ST^z)
\]
is well-defined and holomorphic in the half plane $\T{Re}(z) > 0$.

Now, let $z \in \cc$ with $\T{Re}(z) > 2$ and let $S = (T^2 K^{-2})(1 - T^2)^{-1} \in \C L(\C H_h)$. We note that the operators $S,T$ and $K$ are diagonal operators with respect to the same decomposition of the Hilbert space. In particular they mutually commute and we get that
\[
\begin{split}
\ze_z(\rho(\si_i(x))S)
& = \T{Tr}( T\rho(\si_i(x))T S T^{z-2})
= \T{Tr}\big( (TK) \rho(x) (K^{-1}T) S T^{z-2}  \big) \\
& = \T{Tr}\big( \rho(x) S T^z \big)
= \ze_z(\rho(x) S).
\end{split}
\]
This proves the desired invariance result for the functional $x \mapsto \T{Tr}(\rho(x) R)$ by the uniqueness of holomorphic extensions.

In order to compute the trace of $R = (T^2K^{-2})(1 - T^2)^{-1} T$ we note that $R(\xi^l_{mn}) = q^{2l+1 - 2n}(1 - q^{2l+1})^{-1} q^{l+1/2} \xi^l_{mn}$. In particular we get that
\[
\T{Tr}(R)
= \sum_{l=0,1/2,\ldots}\sum_{n,m=-l}^l q^{2l+1 - 2n}(1 - q^{2l+1})^{-1} q^{l+1/2}
\]
The desired formula now follows from basic computations.
\end{proof}

\begin{theorem}\label{t:ChCo}
The reduced Chern character $\T{Ch}_\phi^1(\C D_q) \in H_\la^1(\C A/\cc,\si_i)$ of the modular spectral triple $\C D_q = (\C A(SU_q(2)),\C H,D_q)$ w.r.t. $\phi$ is represented by the reduced twisted cyclic $1$-cocycle $\La : (x,y) \mapsto C \cd h(b_{\si_i}(x,y))$. Here $C > 0$ is the constant
\[
C := (q^{-1} - q) \T{Tr}(R) = 
\frac{q^{1/2}}{(1 - q^{1/2})^2} + \frac{q^{3/2}}{(1 - q^{3/2})^2}
\]
and $h : SU_q(2) \to \cc$ is the Haar-state.
\end{theorem}
\begin{proof}
By Lemma \ref{l:crep} we have that the reduced Chern character $\T{Ch}_\phi^1(\C D_q) \in H_\la^1(\C A/\cc,\si_i)$ is represented by the reduced twisted cyclic cocycle $\al$. And by Proposition \ref{p:comp} we have that
\[
\al(x,y) = (q^{-1} - q)\T{Tr}(\rho(b_{\si_i}(x,y)) R).
\]
It follows from the invariance under $\si_i \in \au(\C A)$ of the Haar-state and the functional $x \mapsto (q^{-1} - q)\T{Tr}(\rho(x)R)$ (see Lemma \ref{l:inva}) that the map
\[
\Ga : x \mapsto (q^{-1} - q)\T{Tr}(\rho(x) R) 
- (q^{-1} - q) \T{Tr}(R) \cd h(x)
\]
defines a reduced twisted zero cochain $\Ga \in C^0_\la(\C A/\cc,\si_i)$. The coboundary of $\Ga$ clearly agrees with the difference $\al - \La$ and the theorem is proved.
\end{proof}

It is worthwhile to notice that the last theorem can be extended by replacing the Haar-state by \emph{any} linear functional $\be : \C A \to \cc$ with $\be(1) = 1$ and which is invariant under the automorphism $\si_i \in \au(\C A)$.

\begin{remark}
It might also be of interest to study the reduced Chern characters of our modular spectral triple with respect to a different choice of weight $\phi$.
For example, we could replace the Radon-Nikodym derivative $\De$ by a power $\De^s$, for $s > 1$. This change will affect the summability of the modular spectral triple. In particular both the twist and the degree of the associated reduced Chern character are affected. See also \cite[Remark 4.2]{KaaSen:TSQ}. Another possibility is to perturb the weight by powers of the unbounded selfadjoint positive injective operator $\De_R$ defined as the closure of the diagonal operator $\De_{R,0} : \xi^l_{mn} \mapsto q^{2m} \xi^l_{mn}$.
\end{remark}

\subsection{The twisted index pairing}
We end this paper by an investigation of the index pairing between the corepresentation unitaries $u^l = \{t^l_{mn} \} \in U_{2l+1}(\C A)$, $l = 0,1/2,\ldots$ and the modular spectral triple $(\C A,\C H,D_q)$.

\begin{lemma}
Let $l \in \{0,1/2,\ldots \}$. Then the corepresentation unitary $u^l = \{t^l_{mn}\} \in \C U_{2l+1}(\C A)$ satisfies the right modular condition for the modular group of automorphisms $\{\si_t\}$. To be precise, we have that
\[
(u^*)^l\si_z(u^l) = \T{diag}(q^{i l z}, q^{i(l-1)z},\ldots,q^{-i l z}) \in GL_{2l+1}(\cc)
\]
for all $z \in \cc$.
\end{lemma}
\begin{proof}
This follows immediately since $\si_z(t^l_{mn}) = q^{-inz} t^l_{mn}$ for all $z \in \cc$.
\end{proof}

Let $l \in \{0,1/2,\ldots\}$. We then get a twisted index
\[
\T{Ind}_\phi(Pu^lP) = \phi(g_l^{1/2} \cd K_{Pu^lP} \cd g_l^{1/2}) - \phi(K_{P(u^l)^* P} )  
\]
where $P = 1_{[0,\infty)}(D_q)$ denotes the projection onto the positive part of the spectrum of the $q$-Dirac operator $D_q$ and $g_l = \T{diag}(q^{-l}, q^{(-l+1)},\ldots,q^{l})$.

The results of Section \ref{S:TwInPa} and Subsection \ref{s:CoChCh} now allow us to compute this twisted index explicitly.

\begin{prop}
Let $u^l = \{t^l_{mn} \} \in \C U_{2l+1}(\C A)$ denote the corepresentation unitary for some $l \in \{0,\frac{1}{2},\ldots\}$. The twisted index $\T{Ind}_\phi(Pu^lP)$ associated with the corepresentation unitary and the modular spectral triple $(\C A,\C H,D_q)$ is then given by the formula
\[
\T{Ind}_\phi(Pu^lP)
= \frac{1}{2} \inn{\T{Ch}^\phi_1(u^l),\T{Ch}_\phi^1(\C D_q)}
= \frac{C}{2} \cd \big( (2l+1) - [2l+1]_{q^{1/2}} \big)
\]
where $C = \frac{q^{1/2}}{(1 - q^{1/2})^2} + \frac{q^{3/2}}{(1 - q^{3/2})^2}$.
\end{prop}
\begin{proof}
By Proposition \ref{p:PaChFr} and Theorem \ref{t:ChCo} we have the identities
\[
\begin{split}
\T{Ind}_\phi(Pu^lP)
& = \frac{1}{2}\inn{\T{Ch}^\phi_1(F \ot 1_k),[u^l] \ot_{M_k(\cc)} [(u^l)^*]}
= \frac{C}{2} \cd h\big( b_{\si_i}([u^l] \ot_{M_k(\cc)} [(u^l)^*] ) \\
& = \frac{C}{2} h(1_{2l+1} - g_l)
= \frac{C}{2} \big( 2l+1 - \sum_{i = -l}^l q^i \big)
= \frac{C}{2} \big( 2l+1 - [2l+1]_{q^{1/2}} \big).
\end{split}
\]
But this computation proves the proposition.
\end{proof}

\section{Example 3: Podle\'s Sphere}\label{S:pod}
In this section we show that the spectral triple  over the standard Podle\'s sphere satisfies the conditions of a $3$-summable modular spectral triple for an appropriate choice of weight $\psi : \C L(\C H)_+ \to [0,\infty]$. In particular we obtain a reduced twisted index cocycle $\T{Ch}_{\phi}(\C A(S_q^2),\C H,D_q) \in H^2_\la(\C A(S_q^2)/\cc,\si_i^F)$. The twisting automorphism $\si_i^F \in \au(\C A(S_q^2))$ is related to the Haar-state $h : SU_q(2) \to \cc$ by the formula $h(xy) = h(\si_i^F(y)x)$ for all $x,y \in \C A(S_q^2)$. It is then an interesting question to compare our reduced twisted index cocycle (or its non-reduced analog) with the residue cocycle appearing in \cite[Section 4.2]{ReSe:TLP}, see also \cite[Theorem 4.1]{NeTu:LQS}.

Let us start by recalling the ingredients of the spectral triple on the standard quantum sphere as it appears in \cite{DaSi:DPQ} for example.

We note that the quantum spheres $S_q^2$ were introduced by Podle\'s in \cite{Pod:QS}. Let $q \in (0,1)$. The coordinate algebra $\C A(S_q^2)$ is the unital $*$-algebra generated by $A,B\in \C A(S_q^2)$ with $A = A^*$ satisfying the relations
\[
\arr{ccc}{
AB = q^2 BA & & AB^* = q^{-2} B^* A \\
BB^* = q^{-2}A(1- A) & & B^*B = A(1 - q^2 A).
}
\]
The coordinate algebra $\C A(S_q^2)$ can be realized as a $*$-subalgebra of $\C A(SU_q(2))$ by letting $A = 1 - da$ and $B = dc$. The Hopf $*$-algebra structure on $\C A(SU_q(2))$ induces a Hopf $*$-algebra structure on $\C A(S_q^2)$.

We let $\C H_+$ and $\C H_-$ denote the Hilbert spaces obtained by
\[
\C H_+ := \ov{\T{span}\{\xi^l_{mn}\, | \, n = 1/2 \}} \q \T{and} \q 
\C H_- := \ov{\T{span}\{\xi^l_{mn}\, | \, n = -1/2 \}}
\]
where the closures are taken inside the GNS-space $\C H_h$ associated with the Haar-state on $SU_q(2)$. The Hilbert space $\C H := \C H_+ \op \C H_-$ is then a $\zz/(2\zz)$-graded Hilbert space which carries a representation of the quantum sphere $S_q^2$ as even bounded operators,
\[
\rho : S_q^2 \to \C L(\C H) \q \rho(x) = \ma{cc}{\pi(x)|_{\C H_+} & 0 \\
0 & \pi(x)|_{\C H_-}}
\]
Here $\pi : SU_q(2) \to \C L(\C H_h)$ denotes the GNS-representation of quantum $SU(2)$.

We let $\C A_1$ and $\C A_{-1}$ denote the dense subspaces of $\C H_+$ and $\C H_-$ defined by $\C A_1 :=\T{span}\{\xi^l_{mn}\, | \, n = 1/2 \}$  and $\C A_{-1} := \T{span}\{\xi^l_{mn}\, | \, n = -1/2 \}$. We then define the unbounded operator
\[
D_{q,0} := \ma{cc}{0 & E \\ F & 0} \q D_{q,0} : \C A_1 \op \C A_{-1} \to \C H.
\]
We will use the notation $D_q := \ov{D_{q,0}}$ for the closure and refer to this unbounded operator as the \emph{$q$-Dirac operator}.

In order to define the appropriate weight $\phi : \C L(\C H_+) \to [0,\infty]$ we let $\De_R : \sD(\De_R) \to \C H$ denote the closure of the unbounded diagonal operator $\De_{R,0} : \C A_1 \op \C A_{-1} \to \C H$, $\De_{R,0}(\xi^l_{mn}) = q^{2m} \xi^l_{mn}$. The weight $\phi$ is then given by
\[
\phi(T) = \lim_{n \to \infty} \T{Tr}( (\De_R)_{1/n}^{1/2} T (\De_R)_{1/n}^{1/2}) \q T \in \C L(\C H)_+.
\]
See Section \ref{S:der}. The restriction of the modular group of automorphisms $\{\si^F_t\}$ to the algebra $\C A(S_q^2)$ is given by $\si^F_t(\xi^l_{mn}) := q^{it2m} \xi^l_{mn}$, $n = 0$. In particular, we have that $h(xy) = h(\si^F_i(y)x)$ for all $x,y \in \C A(S_q^2)$, see \cite[Proposition 15]{KlSc:QGR}.

The results of \cite[Theorem 8]{DaSi:DPQ} and \cite[Lemma 1]{KrEl:RFP} then imply the following:

\begin{theorem}
The triple $(\C A(S_q^2),\C H,D_q)$ is an even unital $p$-summable Lipschitz regular modular spectral triple w.r.t. the weight $\phi$ for all $p > 2$.
\end{theorem}

We believe that the Hochschild class of the non-reduced version of the Chern character $\T{Ch}^2_\phi(\C A(S_q^2),\C H,D_q) \in H_\la^2(\C A(S_q^2),\si_i^F)$ coincides with the fundamental Hochschild class appearing in \cite{Kra:HSQ} for example. See also \cite{KrEl:RFP}.

\bibliography{JK.bib}
\bibliographystyle{amsalpha-lmp}

\end{document}